\documentclass[11pt,reqno]{amsart}
\usepackage{amssymb, amsmath, amsthm,amsrefs,marvosym,wasysym}
\usepackage{latexsym}
\usepackage{color}
\usepackage{exscale}
\usepackage{fullpage}
\usepackage{rotating}
\usepackage{bm, bbm, mathrsfs}
\usepackage{graphics, graphicx}
\newtheorem{theorem}{Theorem}[section]
\newtheorem{lemma}[theorem]{Lemma}

\newtheorem{remark}{Remark}[section]

\def\eq#1{(\ref{#1})}

\def\({\left(\begin{array}{cccccc}}
\def\){\end{array}\right)}

\def\eq#1{(\ref{#1})}

\def\({\left(\begin{array}{cccccc}}
\def\){\end{array}\right)}

\def\bes{\begin{eqnarray}}
\def\ees{\end{eqnarray}}

\newcommand{\beq}{\begin{equation}}
\newcommand{\eeq}{\end{equation}}
\newcommand{\bea}{\begin{eqnarray}}
\newcommand{\eea}{\end{eqnarray}}
\newcommand{\beann}{\begin{eqnarray*}}
\newcommand{\eeann}{\end{eqnarray*}}

\newcommand{\eps}{\ensuremath{\varepsilon}}

\newcommand{\RR}{\mathbb{R}}

\DeclareMathOperator{\sgn}{sgn}

\DeclareMathAlphabet{\mathpzc}{OT1}{pzc}{m}{it}

\DeclareMathOperator{\grad}{grad}
\DeclareMathOperator{\dv}{div}

\numberwithin{equation}{section}
 
\begin{document}

\title{Non-isentropic cavity flow for the multi-d compressible Euler system}
\begin{abstract}
	We rigorously construct non-isentropic and self-similar multi-d Euler flows in which a 
	central cavity (vacuum region) collapses. While isentropic flows of this type have 
	been analyzed earlier by Hunter \cite{hun_60} and others, the non-isentropic setting introduces 
	additional complications, in particular with respect to the behavior along the fluid-vacuum 
	interface. The flows we construct satisfy the physical boundary conditions:
	the interface is a material surface along which  the pressure vanishes, and it
	propagates with a non-vanishing and finite acceleration until collapse.
	
	We introduce a number of algebraic conditions on the parameters in the problem (spatial dimension,
	adiabatic index, similarity parameters). With these conditions satisfied, a simple argument  
	based on trapping regions for the associated similarity ODEs yields the existence of 
	non-isentropic cavity flows. We finally verify that the conditions are all met for several physically
	relevant cases in both two and three dimensions.


\noindent
\end{abstract}

\author{Helge Kristian Jenssen}
\address{H.\ K.\ Jenssen: Department of
Mathematics, Penn State University,
State College, PA 16802, USA
({\tt jenssen@math.psu.edu})}

\author{Charis Tsikkou}
\address{C.\ Tsikkou: Department of
Mathematics, West Virginia University,
Morgantown, WV 26506, USA 
({\tt tsikkou@math.wvu.edu})}

\date{\today}
\maketitle

\tableofcontents

\section{Introduction}\label{intro}
\subsection{Setup and goal}\label{setup_goal}
The non-isentropic Euler system describes the time evolution of 
a compressible fluid in the absence of viscosity and heat conduction:
\begin{align}
	\rho_t+\dv_{\bf x}(\rho \bf u)&=0 \label{mass_m_d_full_eul}\\
	(\rho{\bf  u})_t+\dv_{\bf x}[\rho {\bf  u}\otimes{\bf  u}]+\grad_{\bf  x} p&=0
	\label{mom_m_d_full_eul}\\
	(\rho \mathcal E)_t+\dv_{\bf x}[(\rho \mathcal E+p){\bf  u}]&=0.\label{energy_m_d_full_eul}
\end{align}
The independent variables are time $t$ and position ${\bf  x}\in\RR^n$, 
and the dependent variables are density $\rho$, fluid velocity ${\bf  u}$, 
and specific internal energy $\eps$; the total energy density is $\mathcal E=\eps+\textstyle\frac{1}{2}|{\bf  u}|^2$. 
We assume the fluid is an ideal polytropic gas for which the pressure $p$ and the internal 
energy $\eps$ satisfy
\beq\label{p_e}
	p(\rho,\eps)=(\gamma-1)\rho \eps\qquad\text{and}\qquad \eps=c_v\theta,
\eeq
where $\gamma>1$ and $c_v>0$ are constants. 
The sound speed $c$ is then given by
\beq\label{sound_speed}
	c=\sqrt{\textstyle\frac{\gamma p}{\rho}}=\sqrt{\gamma(\gamma-1)\eps}\propto\sqrt\theta.
\eeq

In what follows we specialize to radial flows, i.e., solutions to 
\eq{mass_m_d_full_eul}-\eq{energy_m_d_full_eul} 
where the flow variables depend on position only through $r=|{\bf x}|$, 
and the velocity field is purely radial,
\beq\label{vel}
	{\bf u}(t,{\bf x})=u(t,r)\textstyle\frac{\bf x}{r}.
\eeq
Choosing $\rho$, $u$, $c$ as dependent variables, the Euler system 
\eq{mass_m_d_full_eul}-\eq{energy_m_d_full_eul} reduces to
\begin{align}
	\rho_t+u\rho_r+\rho(u_r+\textstyle\frac{(n-1)u}{r}) &= 0\label{m_eul}\\
	u_t+ uu_r +\textstyle\frac{1}{\gamma\rho}(\rho c^2)_r&= 0\label{mom_eul}\\
	c_t+uc_r+{\textstyle\frac{\gamma-1}{2}}c(u_r+\textstyle\frac{(n-1)u}{r})&=0\label{ener_eul}
\end{align}
where $r>0$, $\rho=\rho(t,r)$, $u=u(t,r)$, and $c=c(t,r)$. 

In $C^1$ Euler flow, the specific entropy $S$ (specified up to an additive constant 
via Gibbs' relation $de=\theta dS-pd(\frac{1}{\rho})$, where $\theta$ denotes absolute temperature) 
is transported along particle paths. For $C^1$ radial solutions it follows that
\beq\label{spec_entr}
	S_t+uS_r=0.
\eeq

The system \eq{m_eul}-\eq{mom_eul}-\eq{ener_eul} admits self-similar solutions 
\cites{laz,gud,cf,sed,stan,rj}; we follow the setup in \cites{laz,cf} and posit 
\beq\label{sim_vars}
	x=\textstyle\frac{t}{r^\lambda},\qquad \rho(t,r)=r^\kappa R(x),\qquad 
	u(t,r)=-\frac{r^{1-\lambda}}{\lambda}\frac{V(x)}{x}, \qquad 
	c(t,r)=-\frac{r^{1-\lambda}}{\lambda}\frac{C(x)}{x}.
\eeq
The similarity variables $R(x),V(x),C(x)$ then solve an ODE system recorded in 
\eq{R_sim1}-\eq{C_sim1} below. Note that the signs of $t$ and $x$ agree. 

In this work we are seeking solutions to 
\eq{m_eul}-\eq{mom_eul}-\eq{ener_eul} that are defined for all $t<0$ and all $r\geq0$,
and which contain a shrinking vacuum region of the form $\{-\infty<x\leq x_0\}$, 
where $x_0<0$ is to be identified; see Figure \ref{r_t_plane}. The similarity variables
$R(x),V(x),C(x)$ are to be determined for all $x\in(x_0,0)$ (the fluid region).

\begin{figure}
	\centering
	\includegraphics[width=9cm,height=8.5cm]{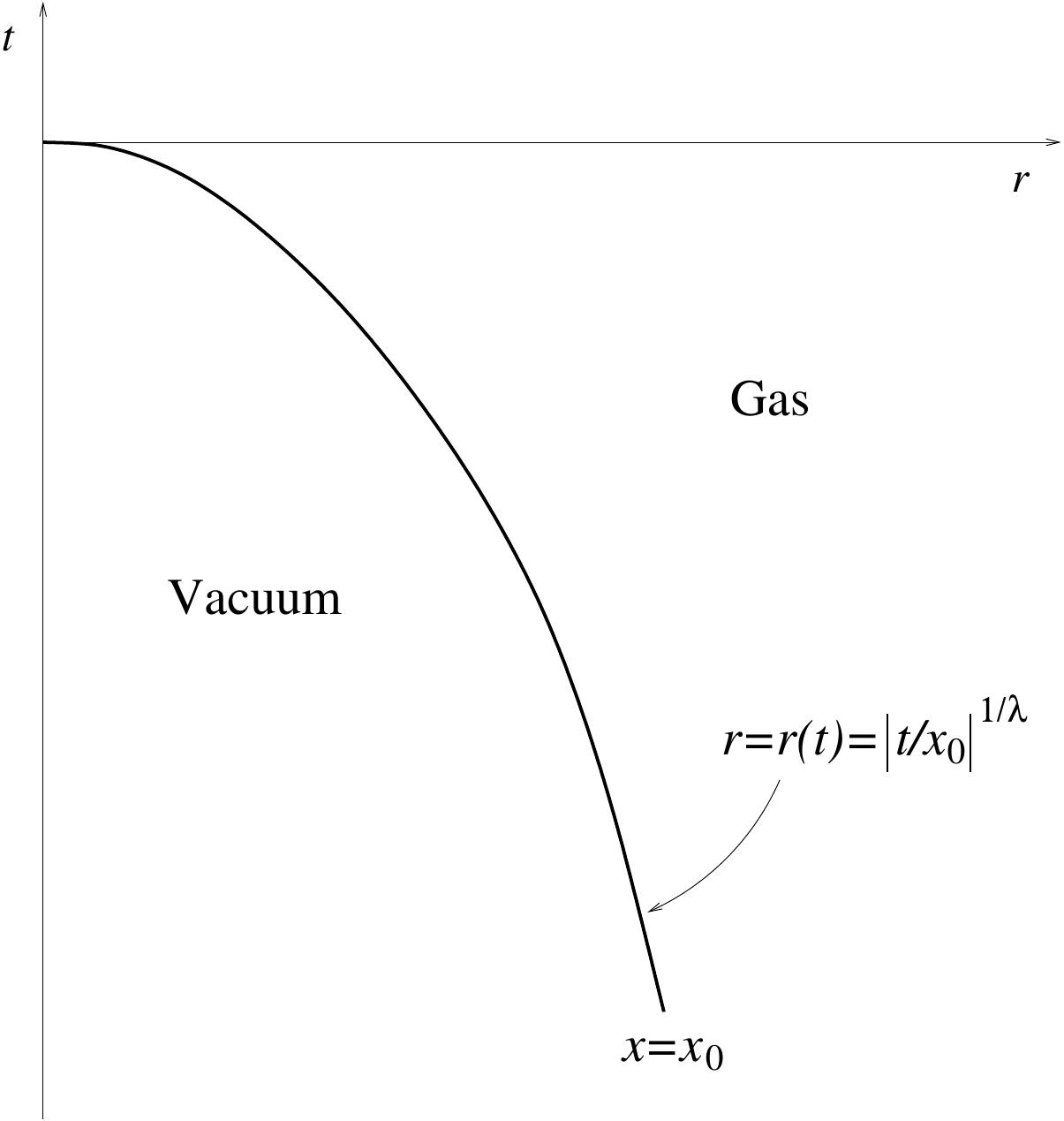}
	\caption{Physical configuration in self-similar cavity flow.}\label{r_t_plane}
\end{figure} 

The similarity parameters $\kappa$ and $\lambda$ are a priori free. We shall exploit this
freedom to impose a number of conditions on the parameters $n,\gamma,\lambda,\kappa$.
Some of these are dictated by physical constraints (integrability, boundary conditions); 
others are imposed to simplify later arguments.

The pioneering work of Guderley \cite{gud} demonstrated that the similarity variables
$V$, $C$ solve a single decoupled ODE (Eqn.\ \eq{CV_ode} below), so that the problem
of constructing radial self-similar Euler flows can be analyzed effectively through a phase
plane analysis. The density variable $R(x)$ is obtained from $V(x),C(x)$
via an exact integral (analyzed in Section \ref{adiabatic_int}). 
In this way Guderley gave examples of solutions containing a 
converging-diverging shock wave and suffering amplitude blowup in $p,c,u$ at time of 
collapse $t=0$.
A mathematically rigorous justification for this type of solutions has been provided only 
recently, \cites{jls1,jls2}.

Employing a similar setup, Hunter \cite{hun_60} (see also \cite{bk}) considered the 
case of a spherical or cylindrical cavity (vacuum region) surrounded by fluid. 
A phase plane analysis demonstrates the existence of cavity flows in which the 
vacuum-fluid interface accelerates and collapses
with infinite speed. This type of solutions were later extended beyond collapse 
in the comprehensive work of Lazarus \cite{laz}, who treated both the shock and cavity 
problems. For the case of cavity flows, the works \cites{hun_60,bk,laz} restricted attention to 
the {\em isentropic} regime where the specific entropy $S$ takes a globally constant value. 

Our goal in this work is to extend the analysis to the {\em non-isentropic} regime and rigorously  
establish the existence of self-similar Euler flows in which a central cavity 
is driven to collapse by a surrounding gas flow with a variable entropy field.

Before starting the analysis proper we review the issue of Euler flows in the presence 
of a vacuum. This issue has attracted considerable attention lately (see references below).

\subsection{Vacuum in Euler flow}\label{vac}
\subsubsection{General pressure law}
Consider a  situation where a gas governed by the Euler equations 
\eq{mass_m_d_full_eul}-\eq{energy_m_d_full_eul} (with a general pressure law $p$)
occupies a dynamically changing region $\Omega(t)\subset\RR^3_{\bf x}$ and 
with a vacuum outside of $\Omega(t)$.

For the radial self-similar solutions considered later, the fluid region will be of the form
\beq\label{our_Omega}
	\Omega(t)=\RR^3\smallsetminus \bar B_{r_0(t)}({\bf 0}),
\eeq
where the radius of the spherical vacuum region is given by
\beq\label{r_0}
	r_0(t)=(\textstyle\frac{t}{x_0})^\frac{1}{\lambda},
\eeq
where $\lambda>1$ and $x_0<0$ are constants, and $t<0$. See Figure \ref{r_t_plane}.
 In this case the 
interface $\partial \Omega(t)$ is a converging sphere which collapses with 
infinite speed at the origin at time $t=0$. Note that determining $x_0$
(really, its finiteness) is part of the problem.

In general $\Omega(t)$ is taken to be an 
open and smooth set and $\rho(t,{\bf x})>0$ whenever 
${\bf x}\in\Omega(t)$. A priori, the only restriction imposed on $\rho$ along
$\partial\Omega(t)$ is that $\rho\geq0$ there. We assume that the 
flow variables $\rho, \bf{u}$ etc.\ are at least $C^1$ within $\cup_t\Omega(t)$.

Two basic boundary conditions are imposed along the vacuum
interface $\partial\Omega(t)$: the pressure should vanish as the interface is approached from 
within the gas (according to the normal stress balance condition),
and the interface is a material surface (consisting of the same fluid particles at all times). 
Thus,
\beq\label{gen_bc}
	p=0\quad\text{and}\quad \mathcal V(\partial\Omega(t))={\bf u}\cdot{\bf n}\qquad
	\text{along $\partial\Omega(t)$},
\eeq
where ${\bf n}$ denotes the exterior unit normal vector along $\partial\Omega(t)$, and 
$\mathcal V(\partial\Omega(t))$ is the normal speed of the surface $\partial\Omega(t)$.
These serve as boundary conditions that must be taken into account in the determination 
of the flow. This is therefore a free-boundary problem where 
determining the location of $\partial\Omega(t)$ is part of the problem.

In general no further boundary conditions are imposed. However, it is clear that in many/most cases 
a vacuum interface would move, due to pressure forces within the fluid, 
with a  {\em finite and non-vanishing normal acceleration}.
This is referred to as the {\em physical boundary condition} (or the 
vacuum free-boundary condition) along the interface.

Denoting the material derivative by $\textstyle\frac{d}{dt}=\partial_t+{\bf u}\cdot\nabla_{\bf x},$
the momentum equation \eq{mom_m_d_full_eul} takes the form
\[\textstyle\frac{d {\bf u}}{dt} +\frac{1}{\rho}\grad_{\bf  x} p={\bf 0}.\]
It follows that the normal acceleration of the interface 
is given by
\beq\label{norm_acc}
	\textstyle\frac{d{\bf u}}{dt}\big|_{\partial\Omega(t)}\cdot {\bf n}
	=-[(\frac{1}{\rho}\grad_{\bf x} p)\big|_{\partial\Omega(t)}]\cdot {\bf n}
	\equiv -(\frac{1}{\rho}\textstyle\frac{\partial p}{\partial n})\big|_{\partial\Omega(t)},
\eeq
where $\frac{\partial}{\partial n}$ denotes the directional derivative along ${\bf n}$.
Since $p\geq0$, $\rho>0$ within $\Omega(t)$ and $p$ vanishes along $\partial\Omega(t)$, 
we have
\beq\label{norm_acc_2}
	\textstyle\frac{d{\bf u}}{dt}\big|_{\partial\Omega(t)}\cdot {\bf n}\geq0.
\eeq
This is physically evident: In Euler flow, and in the absence of external forces, 
the only force is the internal (positive) pressure in the fluid; this will necessarily tend to 
accelerate a vacuum interface, if at all, in the direction of its exterior normal.
It follows from \eq{norm_acc} and  \eq{norm_acc_2} that the physical boundary condition 
amounts to the requirement that
\beq\label{phys_bc}
	-\infty<(\textstyle\frac{1}{\rho}\frac{\partial p}{\partial n})\big|_{\partial\Omega(t)}<0.
\eeq
\begin{remark}
	Strictly speaking the term in the middle of \eq{phys_bc} should be understood 
	in the sense of a limit as $\partial\Omega(t)$ is approached from within
	$\Omega(t)$.
\end{remark}
\begin{remark}
	The physical boundary condition \eq{phys_bc} is formulated only
	in terms of  thermodynamic quantities along the interface. 
	However, as time progresses, these are influenced by all flow variables,
	including the velocity field, within the fluid. See \cite{j_24} for a concrete illustration of the 
	impact of the initial velocity distribution in 1-d (generalized) vacuum Riemann 
	problems.
\end{remark}
\begin{remark}
	To call \eq{phys_bc} \underline{the} physical boundary 
	condition is a little misleading as there are physically relevant flows where 
	the vacuum interface propagates does not accelerate. Examples include 
	vacuum Riemann problems in 1-d \cites{gb,j_24}, as well as 
	certain expanding multi-d flows \cites{gras_98,ser_15}.
\end{remark}

\subsubsection{Ideal gas law}\label{vac_ideal_gas}
We now specialize to the case of an ideal polytropic gas \eq{p_e}. 
According to \eq{sound_speed} we then have $p\propto \rho c^2\propto \rho\eps$,
and there are, a priori, three ways in which the pressure may vanish as the vacuum
interface is approached from within the gas:
\begin{itemize}
	\item[(I)] $\eps\to0$, $\rho\to0$;
	\item[(II)] $\eps\to0$, $\rho\not\to0$;
	\item[(III)] $\eps\not\to0$, $\rho\to0$.
\end{itemize}
It is of interest to also consider the behavior of the specific entropy near vacuum. 
Gibbs' relation implies that the pressure and internal energy are given by
\beq\label{eos_p_e}
	p(\rho,S)=\bar\eps(\gamma-1)\rho^\gamma e^\frac{S}{c_v}\qquad\text{and}\qquad
	\eps(\rho,S)=\bar\eps\rho^{\gamma-1}e^\frac{S}{c_v} 
	\qquad \text{($\bar\eps>0$ constant)}.
\eeq
\begin{remark}[The physical boundary condition in isentropic flow]\label{isntr_vac}
            With $S$ taking a constant value throughout the gas, we have
            $p\propto\rho^\gamma$ and $\eps\propto \rho^{\gamma-1}$, and {\em (I)}
            is the only possible behavior at a vacuum interface. As 
            $\grad_{\bf x}p\equiv \frac{\rho}{\gamma-1} \grad_{\bf x}(c^2)$ in this case, the 
            physical boundary condition \eq{phys_bc} takes the form 
            \beq\label{phys_bc_isntr}
                	-\infty<\textstyle\frac{\partial (c^2)}{\partial n}\big|_{\partial\Omega(t)}<0.
		\qquad\text{(Isentropic)}
            \eeq
            As stressed in \cites{liu_96,mak_86}, \eq{phys_bc_isntr} implies that
            $c$ suffers a square root singularity at an accelerating interface. 
            In turn, this renders local well-posedness a challenging issue that has 
            been analyzed in a number of works; see \cites{cs2,jm,jm2,it_24} for  
            overviews, references, and recent results. We note that the situation when 
            the physical boundary condition is violated by the initial data has not been 
            fully resolved; see \cites{jm_12,ly1,ly}.
\end{remark}

%
We note that Case (I) can also occur in non-isentropic flows; 
examples are provided by non-isentropic affine
flows as analyzed in \cites{sid_17,rhj_21}, and also by the self-similar solutions 
considered in the present work (see Remark \ref{rho_S_bhvrs}). 

The class of affine flows also gives examples of the behavior in Case (II), see \cite{ri_21}. 
Note that, according to \eq{eos_p_e}, the specific entropy $S$ must then 
necessarily tend to $-\infty$ at the interface. On the other hand, 
the total amount of entropy in the fluid (i.e., the 
integral of $\rho(t,{\bf x})S(t,{\bf x})$ over $\Omega(t)$)
might well be bounded, as is the case for the flows in \cite{ri_21}.

We are not aware of examples where Case (III) occurs; again by \eq{eos_p_e},
we must have $S\to+\infty$ at the interface in this case.

Next, consider the physical boundary condition \eq{phys_bc} in non-isentropic flow.
According to \eq{p_e} we have $p=(\gamma-1)\rho\eps$, so that 
\[\textstyle\frac{1}{\rho}\grad_{\bf x}p
\propto\textstyle\frac{\eps}{\rho}\grad_{\bf x}\rho+\grad_{\bf x}\eps,\]
and \eq{phys_bc} takes the form
\beq\label{phys_bc_non_isntr_1}
    	-\infty<\big(\textstyle\frac{\eps}{\rho}\frac{\partial \rho}{\partial n}
	+\frac{\partial \eps}{\partial n}\big)\big|_{\partial\Omega(t)}<0.
\eeq
We note that, while $p\propto \rho\eps$ vanishes along $\partial\Omega(t)$, it is not  
clear what the behavior of the coefficient 
$\frac{\eps}{\rho}\big|_{\partial\Omega(t)}$ 
will be.\footnote{As deduced above, the boundary condition \eq{phys_bc_non_isntr_1} expresses precisely
that the normal acceleration of the 
interface is bounded and non-vanishing. The works \cites{sid_17,rhj_21,ri_21}
formulate the boundary condition somewhat differently; however, the affine 
flows constructed in these works do satisfy \eq{phys_bc_non_isntr_1}.}

Finally, for the self-similar flows we construct below, the physical boundary 
condition \eq{phys_bc} will be satisfied by construction. Indeed, with $\Omega(t)$ 
given by \eq{our_Omega}-\eq{r_0}, it is immediate that $\partial\Omega(t)$
moves with a non-vanishing and finite normal acceleration (prior to collapse).

\section{Outline and main result}\label{outline}
We now return to multi-d self-similar Euler flows of the form \eq{sim_vars},
with ideal pressure law \eq{p_e}, in which a spherical vacuum region $\bar B_{r_0(t)}({\bf 0})$
is invaded by the gas and ``collapses'' at time 
$t=0$. Earlier works \cites{bk,hun_60,laz} provide isentropic cavity
flows; our main interest is to rigorously obtain 
non-isentropic flows of this type. 

As noted in Section \ref{setup_goal}, self-similarity essentially reduces the 
problem to a phase-plane analysis in the $(V,C)$-plane. The challenge is to 
prove the existence of suitable trajectories $x\mapsto (V(x),C(x))$
such that the corresponding flow variables, defined via \eq{sim_vars},
provide a physically meaningful solution of the sought-for type
(in our case, a collapsing cavity). In particular, this entails a 
detailed study of some of the critical points of \eq{CV_ode} in the $(V,C)$-plane.

We shall impose several inter-related conditions 
on the parameters $n, \gamma, \lambda, \kappa$; these are labeled (A)-(J)
below. The conditions are of different types: Some guarantee the basic 
physical constraint of local integrability of conserved quantities, and some are 
dictated by the boundary conditions discussed above. {\em The analysis of the 
latter in the present non-isentropic setting is markedly more involved than in the 
isentropic case} (see Sections \ref{bndry_conds} and \ref{P2}).

Finally, we impose several conditions to ensure that the phase portrait of \eq{CV_ode}
possesses suitable trapping regions. This involves an analysis of the nullclines 
for the reduced similarity ODE and their intersections 
(Sections \ref{GF=0} and \ref{upper_crit_points}).
With all conditions met it is routine to argue that there exist suitable 
ODE trajectories $x\mapsto (V(x),C(x))$ whose corresponding Euler flows
are of the desired type (non-isentropic cavity flow); see Section \ref{conclude}.

The conditions (A)-(J) take the form of algebraic inequalities involving 
$n, \gamma, \lambda, \kappa$. While some of these are straightforward to state, 
others are rather involved. We have not attempted a detailed characterization 
of the allowed values. Instead, some of the conditions are simply stated in 
general terms, and then verified at the end for specific choices of the parameters.
In particular, Section \ref{par_choices} provides physically relevant cases with
$n=2,3$ and $\gamma=\frac{5}{3}, \frac{7}{5},3$ for which all conditions (A)-(J)
are met.

Our main findings are summarized as follows:

\begin{theorem}\label{thm}
	Consider the 2-d and 3-d Euler equations \eq{mass_m_d_full_eul}-\eq{energy_m_d_full_eul}
	for an ideal polytropic gas with adiabatic index $\gamma>1$.
	This system admits non-isentropic, radial and self-similar solutions 
	of the form \eq{sim_vars} with $\lambda>1$, in which a spherical cavity collapses at the center of motion. 
	
	The constructed flows  exhibit blowup in velocity and sound speed at collapse
	while containing locally bounded amounts of mass, momentum, and energy at all 
	times (including time of collapse). The flows are classical solutions within the 
	fluid region whose pressure vanishes along the fluid-vacuum interface. 
	The interface is a material surface and propagates with a non-vanishing 
	and finite acceleration prior to collapse. Finally, flows of this type exist for 
	$\gamma=\frac{5}{3}, \frac{7}{5},3$ in both 2-d and 3-d.
\end{theorem}
\begin{remark}
	By construction the vacuum interface in these flows propagates along a curve 
	$\frac{t}{r^\lambda}\equiv x_0$ (const.\ $<0$). 
	The blowup in $u,c$ is then an immediate consequence
        of our standing assumption that $\lambda>1$ in \eq{sim_vars}: 
        With a fixed $x$-value $\bar x\in(x_0,0)$, \eq{sim_vars} gives 
        $u,c\propto |t|^{\frac{1}{\lambda}-1}$ along the curve $r=\bar r(t)=|\frac{t}{\bar x}|^\frac{1}{\lambda}$;
        as $t\uparrow 0$. In contrast, blowup of the density depends on the sign 
        of the similarity parameter $\kappa$ in \eq{sim_vars}, $\kappa<0$ 
        yielding blowup at time of collapse. The concrete cases considered in 
        Section \ref{par_choices} show that density-blowup may or may not occur in 
        non-isentropic cavity flow.
\end{remark}
\begin{remark}
	As noted in Section \ref{vac_ideal_gas}, the works \cites{sid_17,rhj_21,ri_21}
	provide non-isentropic vacuum flows that satisfy the physical boundary 
	condition \eq{phys_bc_non_isntr_1}. These are quite different from the solutions
	described by Theorem \ref{thm}: The fluid region is a compact set which evolves
	according to affine, or near-affine, velocity field, and no blowup occurs. 
\end{remark}

In the next section we  write out the ODEs for the similarity 
variables $R,V,C$, record the reduced similarity ODE for $\frac{dC}{dV}$, and discuss 
the relevant solutions of the latter. Boundary conditions are discussed in Section 
\ref{bndry_conds}, and this analysis makes essential use of the 
adiabatic integral discussed in Section \ref{adiabatic_int}.

With this background we outline, in Section \ref{strategy}, the various steps 
of the construction of suitable trajectories of the similarity ODEs. 
Repeated reference is made to (the schematic) Figure \ref{V_C_plane} which 
displays the sought-for configuration of level sets and trajectories in the 
$(V,C)$-plane. The details of identifying the sufficient conditions (A)-(J) 
for this to hold are given in Sections \ref{lam_kap_constrs}-\ref{upper_crit_points}.
Section \ref{conclude} concludes the proof of the theorem, modulo the 
conditions (A)-(J). Finally, Section \ref{par_choices} provides several 
concrete choices of $n,\gamma,\lambda,\kappa$ for which all of the conditions
are satisfied.

\section{Self-similar, non-isentropic cavity flows}
\subsection{Similarity ODEs and reduced similarity ODE}\label{ss_odes}
The self-similar ansatz \eq{sim_vars} 
reduces \eq{m_eul}-\eq{mom_eul}-\eq{ener_eul}
to a system of three {\em similarity ODEs} for $R(x)$, $V(x)$, $C(x)$.
A calculation shows that these are given by ($'\equiv\frac{d}{dx}$)
\begin{align}
	(1+V)R'+RV'&=\textstyle\frac{\kappa+n}{\lambda x}RV \label{R_sim1}\\
	C^2R'+\gamma R(1+V)V'+2RCC'&=
	\textstyle\frac{1}{\lambda x}[\gamma(\lambda+V)V+(\kappa+2)C^2]R \label{V_sim1}\\
	\textstyle\frac{\gamma-1}{2}CV'+(1+V)C'&=\textstyle\frac{1}{\lambda x}
	[\lambda+(1+\textstyle\frac{n(\gamma-1)}{2})V]C, \label{C_sim1}
\end{align}
where we observe that \eq{C_sim1} does not involve $R$.
An essential observation, due Guderley \cite{gud}, is that $R$ can be 
eliminated from \eq{V_sim1} by using \eq{R_sim1}. We thus obtain two  
ODEs for only $V$ and $C$. Solving these for $V'$ and $C'$ yields
\begin{align}
	V'&=-\textstyle\frac{1}{\lambda x}\frac{G(V,C)}{D(V,C)}\label{V_sim2}\\
	C'&=-\textstyle\frac{1}{\lambda x}\frac{F(V,C)}{D(V,C)},\label{C_sim2}
\end{align}
with 
\begin{align}
	D(V,C)&=(1+V)^2-C^2\label{D}\\
	G(V,C)&=nC^2(V-V_*)-V(1+V)(\lambda+V)\label{G}\\
	F(V,C)&=C\big[C^2\big(1+\textstyle\frac{\alpha}{1+V}\big)
	-k_1(1+V)^2+k_2(1+V)-k_3\big],\label{F}
\end{align}
where
\beq\label{V_*}
	V_*=\textstyle\frac{\kappa-2(\lambda-1)}{n\gamma},
\eeq
\beq\label{alpha}
	\alpha=\textstyle\frac{1}{\gamma}[(\lambda-1)+\frac{\kappa}{2}(\gamma-1)],
\eeq
and
\beq\label{ks}
	k_1=1+{\textstyle\frac{(n-1)(\gamma-1)}{2}},\qquad 
	k_2={\textstyle\frac{(n-1)(\gamma-1)+(\gamma-3)(\lambda-1)}{2}},\qquad
	k_3=\textstyle\frac{(\gamma-1)(\lambda-1)}{2}.
\eeq
We note that 
\beq\label{sum_k}
	k_1-k_2+k_3=\lambda.
\eeq
Finally, \eq{V_sim2}-\eq{C_sim2} give the single {\em reduced similarity ODE}
\beq\label{CV_ode}
	\textstyle\frac{dC}{dV}=\frac{F(V,C)}{G(V,C)}
\eeq
which relates $V$ and $C$ along self-similar Euler flows. 

Although the phase plane analysis of \eq{CV_ode}
is central to the analysis, at the end of the day we need suitable solutions 
of \eq{V_sim2}-\eq{C_sim2}. This system is not equivalent
to the single ODE \eq{CV_ode}: Trajectories of \eq{CV_ode} cross the 
{\em critical (sonic) lines}
\beq\label{crit_lines}
	L_\pm:=\{C=\pm(1+V)\}
\eeq
smoothly, while \eq{V_sim2}-\eq{C_sim2} degenerates along $L_\pm$.
Direct calculation shows that
\beq\label{non_obvious_reln}
	F(V,\pm(1+V))\equiv \mp\textstyle\frac{(\gamma-1)}{2}G(V,\pm(1+V)),
\eeq
It follows that, if a trajectory $\mathcal T$ of \eq{CV_ode} 
crosses a critical line at a point $P$ where one (and hence both) of $F$ and $G$ 
is non-zero, then the flow of \eq{V_sim2}-\eq{C_sim2} along $\mathcal T$ is directed in 
opposite directions on either side of the critical line at $P$. 
Such a trajectory $\mathcal T$ cannot be used to generate a physically meaningful Euler flow. 

This circumstance severely limits the set of relevant trajectories: Any continuous crossing 
of one of the critical lines $L_\pm$ by a trajectory of  \eq{CV_ode} must occur at 
a ``triple point'' where $F=G=D=0$ in order for the trajectory be relevant as a trajectory 
of the original system \eq{V_sim2}-\eq{C_sim2}. As we shall see below, for the cavity flow problem,
the trajectories of interest must necessarily cross $L_+$.

%

\subsection{Boundary conditions}\label{bndry_conds}
We next discuss the two general boundary conditions that must be met along a vacuum interface:
the interface must be a material surface (i.e., made up of particle trajectories) and
the pressure must vanish within the fluid as the interface is approached. 
We proceed to deduce how, for self-similar radial flows of the form \eq{sim_vars},
these conditions yield starting values for the similarity ODEs \eq{V_sim2}-\eq{C_sim2}.
(Recall that the {\em physical} boundary condition discussed in Section \ref{vac} will be 
automatically satisfied for the self-similar cavity flows under consideration.)

For radial cavity flows of the form \eq{sim_vars}, the fluid-vacuum interface is chosen 
to follow a path $x\equiv x_0$, where $x_0<0$ is to be determined. The interface then
follows $r=r_0(t)=(\textstyle\frac{t}{x_0})^\frac{1}{\lambda}$, and this is a particle trajectory 
provided
\[\dot r_0(t)=u(t,r_0(t)).\]
It follows from \eq{sim_vars}${}_3$ that we must have
\beq\label{V_0}
	V(x_0)=-1.
\eeq
The second condition, i.e., vanishing of the pressure along the fluid-vacuum
interface, amounts to
\beq\label{van_press}
	p(t,r)\to0\qquad\text{as $r\downarrow r_0(t)$ for each fixed $t<0$.}
\eeq
With the understanding that time is held fixed in discussing boundary conditions, 
we write this as: $p\to0$ as $x\downarrow x_0$.

As discussed in Remark \ref{isntr_vac}, in {\em isentropic} flow we have 
$c\to0$ if and only if $p\to0$. Vanishing pressure along the interface 
therefore amounts to $C(x_0)=0$. Recalling \eq{V_0}, we must therefore 
start the ODEs \eq{V_sim2}-\eq{C_sim2}, for a finite $x$-value $x_0<0$, at the point $P_2=(-1,0)$. 
We note that $P_2$ is always a triple point for the system \eq{V_sim2}-\eq{C_sim2} (i.e.,
for any choice of parameters), and that
to analyze its behavior near $P_2$ requires some care. 
(Throughout, we follow \cite{laz} for the labeling of critical points.) Already in the isentropic case,
the ODE \eq{CV_ode} degenerates at $P_2$ with the linearization
\[\textstyle\frac{dC}{dV}=-\textstyle\frac{\gamma-1}{2}\frac{C}{1+V}.\]
This issue is circumvented by instead considering the variable $Z:=C^2$; this change of variable 
removes the degeneracy, and the analysis of the ODE for $\frac{dZ}{dV}$ is straightforward
(see Remark \ref{isntr_trunc_ode}).

For {\em non-isentropic} flow the situation is more involved. As noted in Section 
\ref{vac_ideal_gas}, there could possibly exist non-isentropic cavity flows 
in which $c^2\propto\eps$ does not tend to zero along the interface (Case (III)).
In the present work we avoid this issue by  imposing the same 
conditions on $(V,C)$ as in the isentropic case, i.e., 
\beq\label{bcs}
	V(x_0)=-1\qquad\text{and}\qquad C(x_0)=0.
\eeq
As detailed above, \eq{bcs}${}_1$ implies that the
interface is a material surface.
However, in contrast to the isentropic case, it is not clear what \eq{bcs}${}_2$ implies 
about the pressure along the interface in non-isentropic flow. 
More precisely, $C(x_0)=0$ only gives $c\to 0$ as $x\downarrow x_0$, 
so that, for the non-isentropic case under consideration, we must also analyze the 
behavior of $\rho$ in order to conclude that $p\propto \rho c^2\to0$ as $x\downarrow x_0$.

This requires a detailed analysis of \eq{CV_ode} near the triple point $P_2$, cf.\ Section \ref{P2}. 
For the non-isentropic case this analysis is more involved 
as the degeneracy of \eq{CV_ode} at $P_2$ now is stronger than in the isentropic case.
Specifically, the change of variable $C\mapsto Z=C^2$, while still convenient, does not 
remove the degeneracy in the non-isentropic case. We shall apply classic ODE
results from \cite{ns} to describe the behavior 
near $P_2$. In order to select an appropriate trajectory emanating from $P_2$
(in particular, giving $p\to 0$ as $x\downarrow x_0$), 
we shall make use of an exact integral, which we review next.

\subsection{The adiabatic integral}\label{adiabatic_int}
It is well-known that the similarity ODEs 
\eq{R_sim1}-\eq{C_sim1} admit an 
exact ``adiabatic'' integral (see Eqn.\ (2.7) in \cite{laz}, or pp.\ 319-320 in \cite{rj}).
We briefly review its derivation. 

For an ideal polytropic gas, the specific entropy
$S$ is a function of $\rho^{1-\gamma}\eps\propto \rho^{1-\gamma}c^2$ (cf.\ \eq{eos_p_e}${}_2$). 
For a self-similar flow \eq{sim_vars} it follows that
\beq\label{S_sigma}
	S=\text{function of }(r^{-2\gamma\alpha}\sigma(x))\qquad \text{where}\quad
	\sigma(x):=R(x)^{1-\gamma}(\textstyle\frac{C(x)}{x})^2,
\eeq
and $\alpha$ is given by \eq{alpha}.
The equation \eq{spec_entr} for $S$ therefore becomes
\[(1+V)\sigma'+2\gamma\alpha\textstyle\frac{V}{\lambda x}\sigma=0.\]
Multiplying through by $R$, and substituting for $\frac{RV}{\lambda x}$ from 
the right-hand side of \eq{R_sim1}, result in
\[R(1+V)\sigma'+\textstyle\frac{2\gamma\alpha}{\kappa+n}[R(1+V)]'\sigma=0.\]
It follows that
\beq\label{adiab_int}
	[R|1+V|]^qR(x)^{1-\gamma}(\textstyle\frac{C(x)}{x})^2\equiv const. >0,
\eeq  
where 
\beq\label{q}
	q=\textstyle\frac{2\gamma\alpha}{\kappa+n}\equiv
	\frac{1}{\kappa+n}[\kappa(\gamma-1)+2(\lambda-1)].
\eeq 
Conversely, a calculation verifies that if $(V,C)$ solves \eq{V_sim2}-\eq{C_sim2},
and \eq{adiab_int} holds (defining $R$), then the triple $(R,V,C)$ solves the ODE
system \eq{R_sim1}, \eq{V_sim1}, \eq{C_sim1}. Defining $\rho$, $u$, $c$ 
according to \eq{sim_vars}, we thus obtain a corresponding radial and self-similar
Euler flow.

We stress that \eq{adiab_int}, while obtained using \eq{spec_entr} (i.e., constancy 
of $S$ along particle paths), does not imply that $S$ is constant throughout the fluid. 
Rather, thanks to the integral \eq{adiab_int}, we can {\em characterize} isentropic
self-similar flows as those flows \eq{sim_vars} for which $\kappa$ 
takes the  particular value
\beq\label{isentr_kappa}
	\bar\kappa=\bar\kappa(\gamma,\lambda):=-\textstyle\frac{2(\lambda-1)}{\gamma-1}.
\eeq
Indeed, with $\kappa=\bar\kappa$ the constants $\alpha$ and $q$ vanish, 
and the adiabatic integral \eq{adiab_int} reduces to
\beq\label{CR}
	\sigma(x)=R(x)^{1-\gamma}(\textstyle\frac{C(x)}{x})^2\equiv const.
\eeq
According to \eq{S_sigma} (with $\alpha=0$), \eq{CR} implies that the specific entropy 
$S$ in this case takes a constant value throughout the fluid. 
Conversely, if $S$ is globally constant, then \eq{S_sigma} gives $r^{-2\gamma\alpha}\sigma(x)\equiv const.$;
fixing an $x$-value and varying $r$  gives $\alpha=0$, i.e., $\kappa=\bar\kappa$.


\subsection{Strategy}\label{strategy}
With the background above we can specify more precisely how to obtain 
self-similar Euler flows with a collapsing cavity. Assume for now that we have identified a 
``suitable'' solution $C(V)$ of the reduced similarity ODE \eq{CV_ode}, emanating from the 
critical point $P_2=(V_2,C_2)=(-1,0)$. We then solve \eq{V_sim2}
with $C=C(V)$ to obtain $V= V(x)$ and thus $ C(x):=C( V(x))$.  
As discussed above, the solution will be chosen to satisfy the boundary conditions 
\eq{bcs} for some finite $x_0<0$. 
In order that the corresponding Euler flow be defined for all $t<0$ and all $r>r_0(t)$, the solution 
$(V(x), C(x))$ must be defined for all $x\in(x_0,0)$ (cf.\ \eq{sim_vars}${}_1$). 
Let $\Gamma$ denote the corresponding trajectory, i.e., 
\[\Gamma=\{( V(x), C(x))\,|\, x_0<x<0\}.\]

The analysis in Section \ref{P2} will show that $\Gamma$ 
must start out vertically from $P_2$ in the $(V,C)$-plane; in particular,
it is necessarily located above the critical line $L_+$ near $P_2$. We shall want 
$\Gamma$ to move to the right from $P_2$, and, in order to facilitate later arguments, to
be located between the level sets $\{F=0\}$ and $\{G=0\}$ (both of which contain branches
approaching $P_2$ vertically). See Figure \ref{V_C_plane}.
We shall later detail the conditions ensuring this situation.

\begin{figure}
	\centering
	\includegraphics[width=11cm,height=9cm]{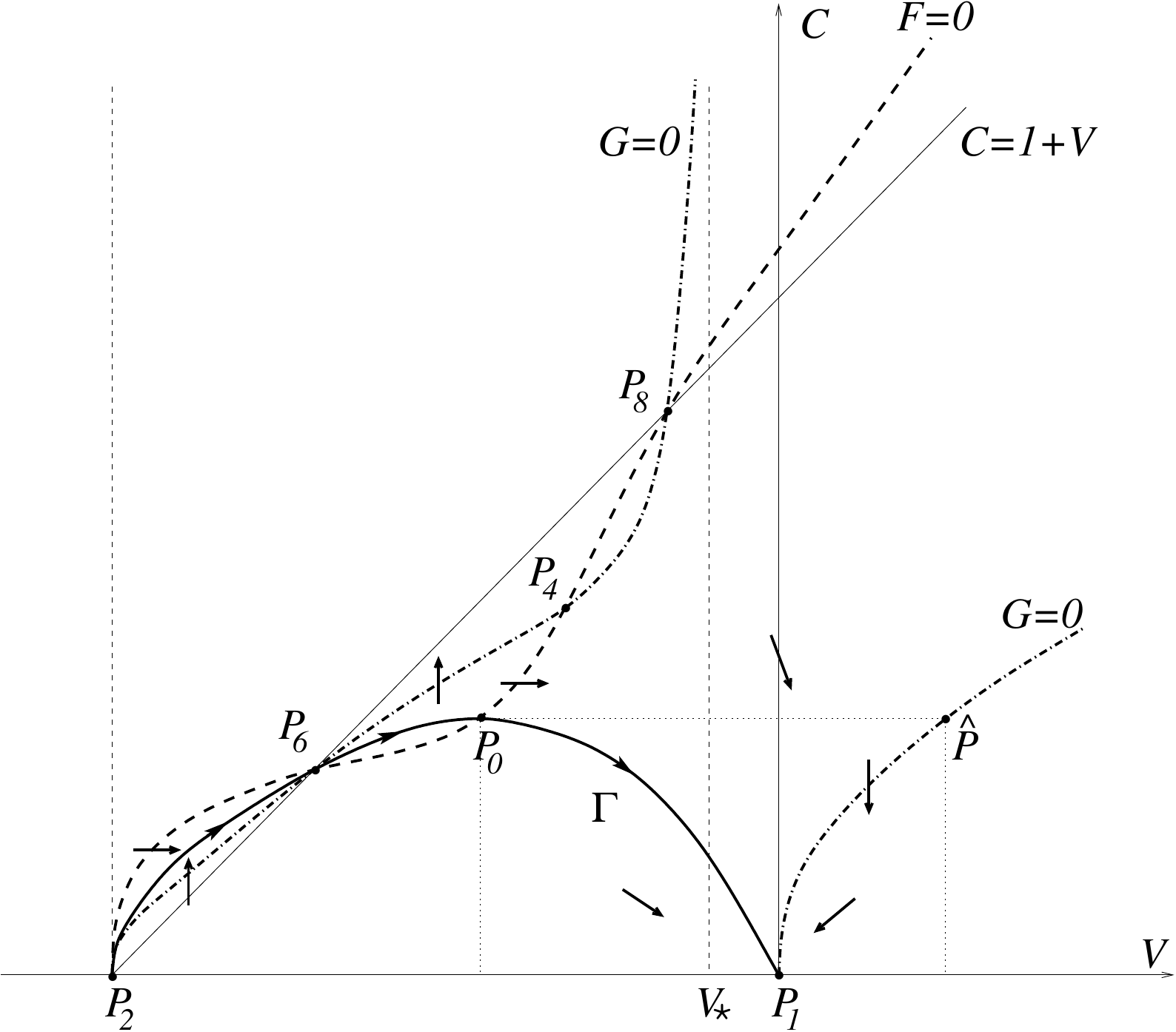}
	\caption{Schematic picture of the sought-for situation in the $(V,C)$-plane. 
	The trajectory $\Gamma$ (solid, thick curve) of a solution to \eq{V_sim2}-\eq{C_sim2} 
	starts out vertically from the point $P_2=(-1,0)$ for $x=x_0<0$ 
	(cf.\ Figure \ref{r_t_plane}).
	As $x$ increases, $\Gamma$ moves to the right, reaching the 
	origin $P_1$ for $x=0$. 
	Within the region $\{V>-1,C>0\}$, $\{F=0\}$ (dashed) consists of a single 
	branch, which is the graph of an increasing function of $V>-1$, while
	$\{G=0\}$ (dash-dot) consists of two branches which are graphs of increasing 
	functions on $(-1,V_*)$ and $(0,\infty)$. The two zero levels intersect 
	at the three critical points $P_4,P_6,P_8$ in the upper half-plane;
	$P_6$ and $P_8$ are located on the critical line $C=1+V$, with $P_6$ being a node.
	Arrows indicate the direction of flow for the original similarity ODEs
	\eq{V_sim2}-\eq{C_sim2} as $x<0$ increases. (As drawn, $\Gamma$ remains within 
	the 2nd quadrant; alternatively, it could cross into the 1st quadrant along
	the $C$-axis. Either way, it must approach $P_1$ as $x\uparrow0$.)
	The rectangle with base along the $V$-axis and vertices $P_0$ and $\hat P$ 
	(dotted)
	will be a trapping region for \eq{V_sim2}-\eq{C_sim2} in which all trajectories 
	tend to the origin $P_1$.}\label{V_C_plane}
\end{figure} 

Next, it follows from the ansatz \eq{sim_vars}${}_3$-\eq{sim_vars}${}_4$
that we must have $ V(0)= C(0)=0$. More precisely, to obtain a meaningful
(pointwise bounded) flow field 
at time $t=0$, we must have that $\frac{ V(x)}{x}$ and $\frac{ C(x)}{x}$ 
approach finite limits as $x\uparrow 0$. It turns out that the critical point $P_1:=(0,0)$
is a star point (proper node) for \eq{CV_ode} (cf.\ Section \ref{P1}), and this will 
imply the existence of these limits.

Since $\Gamma$ is located above $L_+$ near $P_2$ and approaches the origin $P_1$,
it must cross $L_+$. As discussed earlier, this must occur at a triple point where $F=G=D=0$.
We shall impose several conditions to ensure that this happens in a suitable manner.
First of all we need to guarantee that there {\em is} a triple point along $L_+$ in the 
upper half of the $(V,C)$-plane. Moreover, we shall want this to be a node for \eq{CV_ode},
with a suitable location in the second quadrant.


Assuming now that $ V(x), C(x)$ have been defined for all $x\in(x_0,0)$ 
according to the procedure above, we obtain the corresponding Euler flow as follows.
First, the variables $u(t,r)$ and $c(t,r)$ are defined according to \eq{sim_vars}${}_3$-\eq{sim_vars}${}_4$. 
To determine the corresponding density field, we fix a positive constant 
on the right-hand side of \eq{adiab_int}, solve it for $R(x)$, and define 
$\rho(t,r):=r^\kappa R(x)$. As noted in Section \ref{adiabatic_int}, this will then yield a complete 
Euler flow for $t<0$.

Finally, we need to analyze how the pressure $p$ behaves at the vacuum 
interface in the resulting flow. This is done as follows: The analysis of \eq{CV_ode}
will provide the behavior of $(V(x), C(x))$ when $x\approx x_0$ (i.e., near $P_2$); 
using this information in the adiabatic integral \eq{adiab_int} then gives the behavior 
of $R(x)$; finally, the behavior of $p\propto \rho c^2$ as $x\downarrow x_0$ is 
determined from those of $C(x), R(x)$. This analysis will show that we must impose 
a further condition on the parameters to guarantee that $p\to0$ as $x\downarrow x_0$ 
(see condition (B) in Section \ref{P2}). This part of the argument will make use of 
certain results on ODEs with degenerate critical points. The relevant results are 
available in \cites{ns,hart}; see Section \ref{P2} for precise references.

To make all this work out we will impose ten conditions (A)-(J) in the form of 
strict inequalities 
involving quantities which depend on the parameters $n,\gamma,\lambda,\kappa$ in the 
problem. As mentioned, some of these are quite intricate and we do not attempt to
simplify all of them. Instead we verify them for concrete, and physically relevant, choices of 
the parameters. 

With the conditions (A)-(J) imposed we will have good control on the locations
of the relevant critical points, the level sets $\{F=0\}$ and $\{G=0\}$, and hence 
the direction of flow for the original system of similarity ODEs \eq{V_sim2}-\eq{C_sim2}.
Taken together, this will yield a straightforward trapping-region argument 
for the existence of a trajectory $\Gamma$ connecting $P_2$ to $P_1$ 
whose corresponding Euler flow has vanishing pressure along the vacuum interface
$\{x=x_0\}$; see  Section \ref{conclude} for details. This will provide a rigorous proof for 
the existence of non-isentropic and self-similar Euler cavity flows.

The rest of the paper is organized as follows.
In the next section we record the conditions that guarantee 
locally finite amounts of mass, momentum, and energy in the flow.
(Note that, since there will be a vacuum region present near $r=0$ at times $t<0$,
and the flow variables $\rho$, $\eps$, $u$ are locally bounded prior to collapse, these conditions 
need only be imposed at time of collapse $t=0$.)

The critical points $P_1$ and $P_2$ are analyzed in Section \ref{crit_pts}.
Of these, $P_2$ is the more challenging, and here the non-isentropic 
setting differs significantly from the isentropic case. 
It turns out that there are now several trajectories emanating from $P_2$
that could, potentially, serve to build cavity flows. However, the aforementioned 
analysis of the pressure $p$ as $x\downarrow x_0$, will show
that only one of them will give $p\to0$. This particular trajectory is denoted $\Gamma$.

As noted, the curves $\{F=0\}$ and $\{G=0\}$ meet (with infinite slopes) at $P_2$.
In Section \ref{GF=0} we show how to arrange that $\Gamma$ leaves $P_2$ 
between these zero-levels. It will follow from the form of \eq{V_sim2}-\eq{C_sim2}
(for negative $x$) that $\Gamma$ is trapped between $\{F=0\}$ and $\{G=0\}$ near $P_2$.
Under certain conditions, the zero-levels meet again at three critical points, viz.\ $P_4,P_6,P_8$.
In Section \ref{P4_P6_P8} we analyze their presence and locations.
With this information we can impose conditions guaranteeing that the triple point 
$P_6\in L_+$ is the leftmost of these. This implies that $\Gamma$ connects $P_2$ to $P_6$, 
reaching $P_6$ for some finite $x$-value $x_6\in(x_0,0)$.
It only remains to argue that $\Gamma$ can be continued to connect $P_6$ to the origin $P_1$.
To give a simple argument for this, we impose further conditions to guarantee that $P_6$ is a node
(Section \ref{P_6_node}). There will then be an infinite number of trajectories leaving $P_6$ as $x$
increases from $x_6$. It will be convenient to arrange that $\{F=0\}$ and $\{G=0\}$ 
are both given as graphs of increasing functions of $V$. With these properties, it is then 
straightforward to provide a trapping region containing an infinite number of trajectories 
emanating from $P_6$ and ending up at $P_1=(0,0)$ for $x=0$. This will conclude the proof.


\section{Restricting $\lambda$ and $\kappa$ in terms of $n$ and $\gamma$}
\label{lam_kap_constrs}
We recall that we restrict attention to $\lambda>1$. According to \eq{sim_vars}
this has the consequence that the speed $u$ and sound speed $c$ 
suffer amplitude blowup at collapse. 
%
In what follows we also use
\[\mu:=\lambda-1>0\]
as a parameter.
Next, the condition that the flow contains locally finite amounts of mass, momentum, 
and energy requires that the integrals (where $m:=n-1$)
\beq\label{C1}
	\int_0^{\bar r}\rho(t,r)r^m\, dr,\quad
	\int_0^{\bar r}\rho(t,r)|u(t,r)|r^m\, dr,\quad
	\int_0^{\bar r}\rho(t,r)\left(e(t,r)+\textstyle\frac{1}{2}|u(t,r)|^2\right)r^m\, dr
\eeq
are finite at all times $t$ and for all finite $\bar r>0$.
In particular, we require this to hold at time of collapse $t=0$.
Taking for granted that $\frac{ V(x)}{x},\frac{ C(x)}{x}$ 
approach finite limits as $x\uparrow 0$ (this condition is analyzed in Section \ref{P1}),
the ansatz \eq{sim_vars} implies that we must have $\kappa+n>0$,
$\lambda<1+\kappa+n$, and $\lambda<1+\textstyle\frac{\kappa+n}{2}$.
Under the assumption that $\mu=\lambda-1>0$, the two first constraints follows 
from the third, and we thus require 
\begin{itemize}
	\item[(A)] $0<\mu<\textstyle\frac{\kappa+n}{2}$.
\end{itemize}
For later use we note that (A) and $\gamma>1$ give
\beq\label{1+V*}
	0<\textstyle\frac{n+\kappa-2(\lambda-1)}{n\gamma}
	<\textstyle\frac{n\gamma+\kappa-2(\lambda-1)}{n\gamma}\equiv 1+V_*,
\eeq
where $V_*$ is given in \eq{V_*}.

\section{Critical points along the $V$-axis}\label{crit_pts}
\subsection{Preliminaries}\label{prelim_crit_points}
The critical points of \eq{CV_ode} are the common zeros of the functions $F$ and $G$, 
and their number and locations depend on $n,\gamma,\lambda,\kappa$. 
It turns out that some of these are necessarily located
on the critical lines $L_\pm$; they thus provide triple 
points where trajectories of \eq{V_sim2}-\eq{C_sim2} can cross.

Recall that $\sgn(x)=\sgn(t)$ and that we restrict attention to negative times 
in this work. It follows from \eq{sim_vars}${}_4$, together with the convention $c>0$
that the trajectories under consideration are located in the upper half-plane $\{C\geq 0\}$.

A complete list of the critical points of \eq{CV_ode}
for all values of $(n,\gamma,\lambda,\kappa)$ is available in \cite{jj},
and we shall use some of the expressions from that work. 
Employing the same indexing as in 
 \cites{jj,laz}, the critical points come in two groups: the points
$P_1$, $P_2$, $P_3$ located along the $V$-axis, and up to three points $P_4$, $P_6$, 
$P_8$ located in $\{C> 0\}$.

From \eq{G}-\eq{F} we have $F(V,0)\equiv0$ and 
$G(V,0)=V(1+V)(\lambda+V)$. Since we assume $\lambda>1$, it follows that there
are exactly three critical points located along the $V$-axis:
\[P_1=(0,0)\qquad P_2=(-1,0)\qquad P_3=(-\lambda,0).\] 
The critical point $P_3$ will not play any role in what follows;
we proceed to analyze the behavior of \eq{CV_ode} near 
$P_1$ and $P_2$.

\subsection{The critical point $P_1$}\label{P1}
As explained above, the origin $P_1$ will serve as the 
end point of the sought-for trajectory $\Gamma$ connecting 
$P_2$ to $P_1$. 
For any values of $n,\gamma,\lambda,\kappa$, the 
linearization of \eq{CV_ode} at $P_1$ is
\[\textstyle\frac{dC}{dV}=\frac{C}{V},\]
According to \cite{hart} (Theorem 3.5 (iv), p.\ 218), $P_1$ is  
a proper node (a.k.a.\ a star point or dicritical node): 
Every trajectory approaching $P_1$ does so tangent to some 
straight half-line from $P_1$, and for any such half-line 
$\mathcal L$ there is a unique trajectory $(V,C(V))$ of \eq{CV_ode} that 
approaches $P_1$ tangent to it. 

We next analyze how a solution $(V(x),C(x))$ of the system
\eq{V_sim2}-\eq{C_sim2} behaves as $P_1$ is approached tangent to a half-line
$\mathcal L$. Letting $\ell$ denote the slope of $\mathcal L$, we have 
$C\approx \ell V$ when $V\approx 0$
along the solution in question. Substitution into \eq{V_sim2}-\eq{C_sim2}, and 
use of \eq{sum_k}, give
\[\frac{dV}{dx}\approx \frac V x \qquad\text{and}\qquad \frac{dC}{dx}\approx \frac C x
\qquad\text{as $P_1$ is approached.}\]
It follows that any solution to \eq{V_sim2}-\eq{C_sim2} 
approaching $P_1$ does so as $x\to 0$. Furthermore, 
\beq\label{well_bhvd}
	\nu:=\lim_{x\to0}{\textstyle\frac{V(x)}{x}}=V'(0),\qquad
	\omega:=\lim_{x\to0}\textstyle\frac{C(x)}{x}=C'(0),\qquad\text{and}\qquad 
	\ell=\textstyle\frac{\omega}{\nu}.
\eeq
With the notations in \eq{well_bhvd} we can specify  the velocity and sound speed
at time of collapse in the corresponding Euler flow. Sending $t\to0$
with $r>0$ fixed, \eq{sim_vars} and \eq{well_bhvd} give
\beq\label{uc_at_collapse}
	u(0,r)= -\textstyle\frac{\nu}{\lambda}r^{1-\lambda}
	\qquad\text{and}\qquad
	c(0,r)=-\textstyle\frac{\omega}{\lambda} r^{1-\lambda}.
\eeq
In particular, according to (A), both $u$ and $c$ 
suffer amplitude blowup at the center of motion at collapse.

We note that $\omega=0$ gives $\ell=0$, which corresponds to the two flat
trajectories $C(V)\equiv 0$, $V\gtrless0$, of \eq{CV_ode}. Since the trajectories 
under consideration approach $P_1$ from within $\{C\geq0\}$, we have $\omega\leq0$.
On the other hand, $\nu$ may be of either sign (or vanish) depending on whether the trajectory 
approaches the origin $P_1$ from within the 1st or 2nd quadrant (or vertically).

\subsection{The critical point $P_2$}\label{P2}
With the boundary conditions \eq{bcs} imposed, the critical point $P_2$ corresponds to the 
vacuum interface, and the behavior of trajectories of \eq{CV_ode} 
near $P_2$ is therefore central to building non-isentropic cavity flows.

To analyze the behavior of \eq{CV_ode} near $P_2$ we first change variables to
\[W:=1+V\qquad\text{and}\qquad Z:=C^2.\]
In particular, $P_2$ corresponds to the origin in the $(W,Z)$-plane. 
In terms of $W,Z$ the reduced similarity ODE \eq{CV_ode} takes the form
\beq\label{ZW_ode}
	\textstyle\frac{dZ}{dW}=\frac{2Z[\alpha Z-k_3W+W(Z-k_1W^2+k_2W)]}
	{W[\mu W-nW_*Z+W(nZ-W^2-(\mu-1)W)]},
\eeq
where 
\beq\label{W_star}
	W_*:=1+V_*\qquad\text{and}\qquad\mu=\lambda-1.
\eeq
The origin $P_2$ is a degenerate critical point for \eq{ZW_ode}, and the following 
analysis uses results from Sections 4.2-4.3 in \cite{ns} (pp.\ 81-88)
(similar results are provided in Section 8.4 in \cite{hart}). The issue is to determine the phase 
portrait of \eq{ZW_ode} near $P_2=(0,0)$, and in particular the trajectories approaching $P_2$,
together with their leading order behavior there. 

We first record the truncated version of \eq{ZW_ode} 
in which only quadratic terms are retained, viz.
 \beq\label{ZW_ode_trunc}
	\textstyle\frac{dZ}{dW}=\frac{2Z[\alpha Z-k_3W]}{W[\mu W-nW_*Z]}.
\eeq
According to the analysis on pp.\ 85-87 of \cite{ns}, since the truncated ODE 
\eq{ZW_ode_trunc} differ from the original ODE \eq{ZW_ode} only with terms 
of degree three or higher, their phase portraits are qualitatively the same near $P_2=(0,0)$.
In particular, the directions along which a trajectory can approach $P_2$ agree. 
For the homogeneous ODE \eq{ZW_ode_trunc} these are readily calculated; 
referring to Section 4.2 in \cite{ns} we only record the results: 
\eq{ZW_ode_trunc} has exactly three integral (half) rays at the origin within the 
closed 1st quadrant of the $(W,Z)$-plane; furthermore, these are given by the polar angles
\[\phi_1=0,\qquad \phi_2=\textstyle\frac{\pi}{2}, \qquad\text{and}\qquad 
\phi_3=\arctan(\sigma),\]
where
\beq\label{sigma}
	\sigma:=\textstyle\frac{\gamma\mu}{\kappa+n}.
\eeq
Note that $\phi_3\in(0,\frac{\pi}{2})$ since we assume $\mu=\lambda-1>0$ 
and $\kappa+n>0$ (cf.\ condition (A)).

Still for the truncated ODE  \eq{ZW_ode_trunc}, we next want to determine the behavior 
of trajectories near the integral rays $\phi=\phi_1,\phi_2,\phi_3$. 
For  this we employ the following terminology from \cite{ns}.
If nearby trajectories, parametrized by $\phi(r)$ (where $r$ for now denotes the 
radial polar coordinate in the $(W,Z)$-plane), diverge from $\phi_i$ as $r\to0$,
then the half-ray $\phi=\phi_i$ is called an {\em isolated} ray; if trajectories on both sides 
converge to $\phi_i$ as $r\to0$, it is called a {\em nodal} ray. Theorem 4.22 in \cite{ns} 
provide sufficient conditions for these cases.
Applying this theorem, routine calculations (omitted) yield the following for the truncated 
ODE \eq{ZW_ode_trunc}: 
 \begin{enumerate}
	\item $\phi_1=0$ is necessarily an isolated ray;
	\item $\phi_2=\frac{\pi}{2}$ is an isolated ray provided $\kappa>\bar\kappa$, 
	and it is a nodal ray provided $\kappa<\bar\kappa$ (recall that we consider non-isentropic 
	self-similar flows, so that $\kappa\neq\bar\kappa$, cf.\ Section \ref{adiabatic_int});
	\item $\phi_3=\arctan(\sigma)$ is an isolated ray provided 
	$\mu<\frac{n}{2}(\gamma-1)$, and it is a nodal ray provided $\mu>\frac{n}{2}(\gamma-1)$.
\end{enumerate}
First consider the isolated ray $\phi_1=0$; according to part (e) of Theorem 4.326 in \cite{ns} (pp.\ 87-88),
the original ODE  \eq{ZW_ode} has a unique trajectory approaching $P_2$ tangent to 
the ray $\phi=\phi_1$. However, this is the flat trajectory $Z\equiv0$ (i.e., $C\equiv 0$),
which is irrelevant for the construction of cavity Euler flows.

Next, consider the rays $\phi=\phi_2$ and $\phi=\phi_3$. For each of these, 
and regardless whether it is an isolated or a nodal ray of the truncated ODE \eq{ZW_ode_trunc}, 
parts (d) and (e) of Theorem 4.326 in \cite{ns} guarantee the existence of at least 
one trajectory of  \eq{ZW_ode} which approaches $P_2$ tangent to it.
The precise manner in which $P_2$ is approached along the latter 
trajectories is crucial: only cases in which the pressure $p$ in the corresponding Euler flow 
vanishes as $P_2$ is approached are of physical interest (cf.\ \eq{van_press}).

To analyze this issue we make use of the adiabatic integral \eq{adiab_int}.
We first observe that, if $(W(x),Z(x))$ is a trajectory approaching 
$P_2=(0,0)$ tangent to $\phi=\phi_2$ or $\phi=\phi_3$, then $P_2$ is reached for a 
{\em finite} $x$-value $x_0\in(-\infty,0)$.
(Recall that the present analysis applies to times prior to collapse, so that $t$ and
$x$ take negative values, cf.\ \eq{sim_vars}${}_1$.) To argue for this we consider the ODE for 
$W\equiv1+V$. Using \eq{V_sim2} with $y=\log |x|$ as independent variable, we have
\beq\label{W_x}
	\textstyle\frac{dW}{dy}=-\frac{1}{\lambda}K(W,Z),
\eeq
where
\[K(W,Z)=\textstyle\frac{nZ(W-W_*)-W(W-1)(W+\mu)}{W^2-Z}.\]
In either case, i.e., whether the approach to $P_2$ is vertical ($\phi=\phi_2$, $Z\gg W$ near $P_2$) or 
with a finite slope ($\phi=\phi_3$, $Z\approx \sigma W$ near $P_2$), $K(W,Z)$ tends to a 
non-zero, constant value as $(W,Z)\to P_2$. 
It then follows from \eq{W_x} that $P_2$ is reached for a finite $y$-value, and thus 
for a finite $x$-value $x_0$. 
\begin{remark}
	This $x_0$ gives the location of the vacuum interface 
	in the resulting Euler cavity flow: In the $(r,t)$-plane the interface 
	propagates along $r_0(t)$ given by \eq{r_0}.
\end{remark}
Next recall the adiabatic integral in \eq{adiab_int}, i.e.,
\beq\label{ad_int_WZ}
	R^{1-\gamma+q}W^q \textstyle\frac{Z}{x^2}\equiv const.,
\eeq
where $q$ is given in \eq{q}.
Consider first the case of a trajectory approaching $P_2$ tangent to the 
half-ray $\phi=\phi_3$, so that $Z=O(1)W$ along it near $P_2$.
Since $x\downarrow x_0$ as $P_2$ is approached, \eq{ad_int_WZ} together with finiteness of $x_0$
give
\beq\label{R_apprch}
	R=O(1)W^\frac{q+1}{\gamma-1-q}\qquad\text{as $W\to0+$.}
\eeq
In turn this implies that 
\beq\label{p_apprch}
	p\propto \rho c^2=O(1)RZ=O(1)W^\frac{\gamma}{\gamma-1-q}\qquad\text{as $W\to0+$.}
\eeq
Thus, in order to have $p\to0$ as $x\downarrow x_0$, we must require
\[\gamma-1>q=\textstyle\frac{1}{\kappa+n}[\kappa(\gamma-1)+2\mu].\] 
Using that $\kappa+n>0$ (by (A)), the last requirement amounts to
\begin{itemize}
	\item[(B)] $\mu<\textstyle\frac{n}{2}(\gamma-1)$.
\end{itemize}
Note that according to (3) above, this implies that $\phi=\phi_3$ is an isolated ray 
for the truncated ODE  \eq{ZW_ode_trunc}. According to part (e) of Theorem 4.326 in \cite{ns},
there is then a {\em unique} trajectory of the original ODE \eq{ZW_ode} which approaches $P_2$
tangent to the ray $\phi=\phi_3$.

Summing up the discussion above we have: 

\begin{lemma}\label{isol_ray}
Whenever the requirements {\em (A)}-{\em (B)} are met
there is a unique trajectory $\Gamma$ of \eq{ZW_ode} approaching $P_2=(0,0)$ 
with the positive slope $\sigma=\textstyle\frac{\gamma\mu}{\kappa+n}$
in the $(W,Z)$-plane Furthermore, the pressure  
in the corresponding Euler flow tends to zero at the vacuum interface.
\end{lemma}

\begin{remark}\label{rho_S_bhvrs}
	For the flow described in Lemma \ref{isol_ray} it follows from \eq{R_apprch}, 
	together with \eq{sim_vars}${}_2$ and finiteness of $x_0$, 
	that $\rho\to 0$ at the vacuum interface if and only if 
	$0<q+1<\gamma$. Here, the right-most inequality is the condition {\em (B)}, 
	while the left-most inequality will be a consequence of condition {\em (C)} below
	(see Lemma \ref{rel_posns}; in fact, with condition {\em (A)} in force, {\em (C)} amounts to $q>0$). 
	Thus, in all Euler cavity flows constructed in this work, the density decays to zero at 
	the vacuum interface.
	
	Finally, for the entropy field $S$, it follows from  \eq{eos_p_e}${}_1$, \eq{R_apprch}, and \eq{p_apprch} that 
	\beq\label{S_apprch}
		e^\frac{S}{c_v}=O(1)W^{-\frac{\gamma q}{\gamma-1-q}}\qquad\text{as $W\to0+$.}
	\eeq
	Conditions {\em (A)-(C)} will therefore imply that that the specific entropy $S$ necessarily 
	tends to $+\infty$ as the interface is approached form within the fluid. On the other hand,
	a calculation shows that the total amount of entropy 
	is locally bounded near the interface at all times (including at collapse), i.e.,
	\[\int_{r_0(t)}^{\bar r}\rho(t,r) S(t,r)\, r^{n-1}dr<\infty \]
	whenever $r_0(t)<\bar r<\infty$ and $t\leq 0$.
\end{remark}

Finally, consider the case when a trajectory approaches $P_2=(0,0)$ in the 
$(W,Z)$-plane tangent to the vertical half-ray 
$\phi=\phi_2=\frac{\pi}{2}$. We shall argue that the pressure in the corresponding 
Euler flow then {\em fails} to satisfy $p\to0$ as $x\downarrow x_0$. Consequently, this 
case is irrelevant for building Euler cavity flows. 

First, while $\phi=\phi_2$ is an integral ray for the truncated ODE \eq{ZW_ode_trunc},
there are cases in which any trajectory of the original ODE \eq{ZW_ode} which approaches
$P_2$ tangent to $\phi=\phi_2$ does so from the left (i.e., from within $\{W<0\}$). 
An inspection of the phase portrait of \eq{ZW_ode} reveals that such a trajectory cannot cross
$\{W=0\}$ and reach $P_1$, and is therefore not relevant for our purposes. 

Next, assume that the ODE \eq{ZW_ode} has a trajectory approaching $P_2$ vertically 
and from within $\{W>0\}$. We then have $Z\gg W$ as $W\to0+$, and \eq{ZW_ode} shows that the 
leading order behavior along the trajectory is given by
\beq\label{a_vert}
	 Z=O(1)W^a\qquad\text{where}\qquad 
	a=-\textstyle\frac{2\alpha}{nW_*}= \textstyle\frac{2\mu+\kappa(\gamma-1)}{2\mu -\kappa-n\gamma}.
\eeq
Recalling the adiabatic integral \eq{ad_int_WZ} and arguing as above, we get that the pressure in the 
corresponding Euler flow, to leading order, satisfies
\[p=O(1) RZ=O(1) W^b\qquad\text{as $W\to0+$, where}\qquad
b=\textstyle\frac{a\gamma+(1-a)q}{\gamma-1-q}.\]
To have $p\to0$ as $W\to0$ we need $b>0$;
however, the $a$-value in \eq{a_vert} gives $b=0$. This shows that the 
$(W,Z)$-trajectory under consideration gives an Euler flow in which the 
pressure fails to satisfy the boundary condition $p=0$ along the vacuum interface.

We conclude that, in seeking a non-isentropic, self-similar Euler flow with a vacuum
interface propagating along the curve $x\equiv x_0\in(-\infty,0)$ prior to collapse, and in which approach to vacuum corresponds
to the $(V,C)$-trajectory approaching $P_2=(-1,0)$, the only possibility is to use the unique trajectory $\Gamma$
described in Lemma \ref{isol_ray} above.

\begin{remark}\label{isntr_trunc_ode}
	In the isentropic case $\kappa=\bar\kappa$, i.e., $\alpha=0$ (see Section \ref{adiabatic_int}),
	the truncated ODE simplifies to 
	\[\textstyle\frac{dZ}{dW}=-\frac{2k_3Z}{\mu W-nW_*Z},\]
	whose analysis is straightforward:  The origin is a 
	saddle point (for $\mu>0$) with two approaching half-ray trajectories within the closed 1st 
	quadrant. These trajectories are the half-rays $\phi=\phi_1$ and $\phi=\phi_3$ identified 
	above (in particular, approaching $P_2$ vertically in the $(W,Z)$-plane 
	is not an option in the isentropic case). 
	The former gives the, for us irrelevant, flat trajectory $\{Z\equiv0\}$, and the other has slope 
	\[\textstyle\frac{dZ}{dW}\equiv\frac{\gamma\mu}{\kappa+n}\big|_{\kappa=\bar\kappa}
	=\frac{\gamma(\gamma-1)\mu}{n(\gamma-1)-2\mu},\]
	cf.\ \cite{laz} (Section 3, p.\ 320). 
	General ODE results \cites{ns,hart} imply that the original ODE \eq{ZW_ode}
	has a trajectory, unique in the upper half-plane $\{Z>0\}$ and approaching $P_2$ with the 
	same slope. This is precisly the trajectory used in \cites{bk,hun_60,laz} to build 
	isentropic cavity flows. 
\end{remark}

The foregoing analysis has used the boundary condition of vanishing pressure 
to identify the unique trajectory $\Gamma$ of \eq{CV_ode} that must be used in generating 
a self-similar cavity flow. The issue now is whether this $\Gamma$ connects 
$P_2$ to the origin $P_1$ in the $(V,C)$-plane, so that it prescribes a complete flow 
for all $t<0$. To analyze this we consider next
how $\Gamma$ is located with respect to the zero-levels of $F$ and $G$.

\section{The level sets $\{G=0\}$ and $\{F=0\}$ and their locations}\label{GF=0}

\subsection{Relative locations of $\Gamma$, $\{G=0\}$, and $\{F=0\}$ near $P_2$}
As noted earlier, we shall want the trajectory $\Gamma$ identified above
to emanate from $P_2$ between the zero levels of $F$ and $G$ (see Figure \ref{V_C_plane}). 
It is convenient to use $(W,Z)$-coordinates so that
the zero levels  are given by (cf.\ \eq{G}, \eq{F})
\beq\label{f(W)}
	Z=f(W):=\textstyle\frac{W(k_1W^2-k_2W+k_3)}{W+\alpha}
\eeq
and 
\beq\label{g(W)}
	Z=g(W):=\textstyle\frac{W(W-1)(W+\mu)}{n(W-W_*)},
\eeq
respectively. 
\begin{lemma}\label{rel_posns}
	Assume the conditions {\em (A)-(B)}, and also
	\begin{itemize}
		\item[(C)] $\kappa>\bar\kappa$ (equivalently, $\alpha>0$),
	\end{itemize}
	are all met. Then the following inequalities hold
	\beq\label{ineq_1}
		0<g'(0)<\sigma<f'(0),
	\eeq
	where $\sigma=\frac{\gamma\mu}{\kappa+n}$ 
	is the slope of $\Gamma$ at $P_2$ in the $(W,Z)$-plane.
	
	Under conditions {\em (A)-(C)} we therefore have a situation in the $(V,C)$-plane where, near $P_2$, 
	$\{F=0\}$ is located to the right of $V=-1$ and above $\Gamma$, which in turn 
	is located above $\{G=0\}$ (see Figure \ref{V_C_plane}).
\end{lemma}
\begin{proof}
	Using the expressions above for $f$ and $g$ we have that \eq{ineq_1}
	amounts to
	\beq\label{ineq_2}
		0<\textstyle\frac{\mu}{nW_*}<\sigma<\textstyle\frac{k_3}{\alpha}.
	\eeq
	We treat each inequality in turn, from left to right. As $\mu>0$ the first inequality 
	amounts to $W_*=1+V_*>0$, which was shown in Section \ref{lam_kap_constrs}
	to be a consequence of (A) (cf.\ \eq{1+V*}).
	For the second inequality we substitute from \eq{W_star}, \eq{V_*}, and \eq{sigma}, 
	and use $W_*>0$ together with (A) to rearrange it to $\mu<\frac{n(\gamma-1)}{2}$, which is (B).
	Finally, for the rightmost inequality, we first observe that condition (C) amounts to $\alpha>0$
	(cf.\ \eq{alpha} and \eq{isentr_kappa}).
	Together with (A) we therefore have that the third inequality amounts to
	$\gamma\mu\alpha<(\kappa+n)k_3$. Substituting from \eq{alpha} and \eq{ks},
	and rearranging, we again obtain $\mu<\frac{n(\gamma-1)}{2}$, which is (B).
\end{proof}
As a consequence of Lemma \ref{rel_posns}, both $\{F=0\}$ and $\{G=0\}$
approach $P_2$ vertically in the $(V,C)$-plane, with $\Gamma$ located between them.
In order to set up a situation where it is simple to argue that, first, $\Gamma$ reaches a
triple point on the critical line $L_+=\{C=1+V\}$, and second,
continues on to reach the origin $P_1$, it will be convenient that both of the zero-levels
$\{F=0\}$ and $\{G=0\}$ are graphs of increasing functions of $V$. These requirements are 
analyzed in the next two subsections.

\subsection{The level set $\{F=0\}$}\label{F=0} 
%
%
%
%
We have that $\{F=0\}$ is the graph of an increasing function of $V$ if and only if
the function $f(W)$ in \eq{f(W)}  is increasing for $W>0$.
To analyze this we assume, in addition to (A)-(C), the condition that:
\begin{itemize}
	\item[(D)] $k_2>0$, i.e., 
	\[(n-1)(\gamma-1)+(\gamma-3)\mu>0.\]
\end{itemize}
According to \eq{f(W)} we have $f'(W)>0$ if and only if
\[\varphi(W):=2k_1W^3+(3\alpha k_1-k_2)W^2-2\alpha k_2W+\alpha k_3>0.\]
Noting that the leading coefficient $2k_1$ is positive (cf.\ \eq{ks}), we have that $\varphi(W)$ 
is increasing on $(-\infty,W_-)\cup(W_+,\infty)$ and decreasing on $(W_-,W_+)$,
where, according to conditions (C) and (D),
\[W_-=-\alpha<0\qquad\text{and}\qquad W_+=\textstyle\frac{k_2}{3k_1}>0\]
are the roots of $\varphi'(W)=0$. In particular, the local maximum point $W_-$ 
lies to the left of the origin, while the local minimum point $W_+$ lies to the right 
of the origin. It follows that
$\varphi(W)>0$ for all $W>0$ if and only if  $\varphi(W_+)>0$.
We include this as a further condition to be checked:
\begin{itemize}
	\item[(E)] $\varphi(W_+)>0$, i.e.,
	\beq\label{phi_min}
		-\textstyle\frac{k_2^3}{27k_1^2}-\frac{\alpha k_2^2}{3k_1}+\alpha k_3>0.
	\eeq
\end{itemize}
We make one more observation about 
the zero-level set of $F$. In $(V,C)$-coordinates this is given by
\[C^2=\textstyle\frac{(1+V)(k_1(1+V)^2-k_2(1+V)+k_3)}{1+\alpha+V},\]
and the argument above shows that, under conditions (A)-(E), this defines a graph $C=C_F(V)$
where $C_F(V)$ is an increasing function of $V$ for $V>-1$. We note that, since $\alpha>0$
by (C),
the function $C_F(V)$ is defined for all $V>-1$. In particular, $\{F=0\}$ has of a single 
branch within the set $\{(V,C)\,|\, V>-1, C>0\}$ (cf.\ Figure \ref{V_C_plane}).

\subsection{The level set $\{G=0\}$}\label{G=0} 
In this section we analyze the level set $\{G=0\}$ with 
the goal of giving conditions guaranteeing that it is the graph of an increasing function 
of $V$. 
From \eq{G} we have that $G(V,C)=0$ if and only if 
\beq\label{C_sqrd_G}
	C^2=\textstyle\frac{V(1+V)(\lambda+V)}{n(V-V_*)}.
\eeq
Recall that $\lambda>1$ by condition (A), and that \eq{1+V*} 
gives $W_*=1+V_*>0$. In what follows we impose three
additional constraints which are collected in condition (F) below.
The first of these is that $V_*<0$, or according to \eq{V_*}: $\kappa<2\mu$.
It follows that $-\lambda<-1<V_*<0$, so that $\{G=0\}$ has three branches located within 
$\{V\leq-\lambda\}$, $\{-1\leq V< V_*\}$, and $\{V\geq0\}$, respectively. 

The behavior of $\{G=0\}$ for 
$V<-\lambda$ is irrelevant for our purposes, and we proceed to give further conditions 
guaranteeing that $\{G=0\}$ is given as the graph of an increasing function of $V$
when $V\in[-1,V_*)\cup[0,\infty)$. Note that $\{G=0\}$ has a vertical asymptote at $V=V_*$.

Switching again to the variable $W=1+V$, this amounts to specifying conditions guaranteeing that 
the function $g(W)$ in \eq{g(W)} is an increasing function of $W$ for all $W\in[0,W_*)\cup [1,\infty)$.
A direct calculation gives that this is the case if and only if 
\[\psi(W):=2W^3+(\mu-1-3W_*)W^2-2(\mu-1)W_*W+\mu W_*>0\qquad \text{for all $W\in[0,W_*)\cup [1,\infty)$.}\]
$\psi(W)$ is a cubic polynomial with a positive leading coefficient
and with the property that its derivative has two real roots: $\psi'(W)$ vanishes for 
$W=\frac{1}{3}(1-\mu)$ and $W=W_*$. For concreteness we assume
that $\frac{1}{3}(1-\mu)<W_*$ (included in condition (F) below). Under this assumption we thus have:  
$g'(W)>0$ for all $W\in[0,W_*)\cup [1,\infty)$ if and only if $\psi(W_*)>0$. 

According to the analysis above we impose the following condition:
\begin{itemize}
	\item[(F)] $\kappa<2\mu$, $\frac{1}{3}(1-\mu)<W_*$, and $\psi(W_*)=W_*(-W_*^2+(1-\mu)W_*+\mu)>0$.
\end{itemize}
With (F) met we therefore have: the part of $\{G=0\}$ located within $\{(V,C)\,|\, V\geq-1, C>0\}$
consists of two branches which are given as an increasing function $C=C_G(V)$ defined 
for $V\in[-1,V_*)\cup[0,\infty)$; furthermore $C_G(V)\uparrow+\infty$ as $V\uparrow V_*$.

It follows from this last property that the level sets $\{F=0\}$ and $\{G=0\}$ must intersect
at one or more points within the set $\{(V,C)\,|\, -1<V<V_*,C>0\}$ (these points will be determined 
explicitly in Section \ref{P4_P6_P8}). Let the leftmost
of these be denoted $\bar P=(\bar V,\bar C)$ for now (in Figure \ref{V_C_plane}, this is $P_6$). 
Thus, as we move rightward from $P_2$, both of the level sets $\{F=0\}$ and $\{G=0\}$  start out 
from $P_2$ and then intersect again at the point $\bar P=(\bar V,\bar C)$, where $-1<\bar V<V_*$.

We next argue that, with the conditions (A)-(F) met, the region bounded by $\{F=0\}$ and $\{G=0\}$
within the strip $-1<V<\bar V$ is a trapping region for the system of similarity ODEs
\eq{V_sim2}-\eq{C_sim2}  (when $x<0$).

\subsection{First trapping region}\label{1st_trap_regn} 
With conditions (A)-(F) met, we have that $\Gamma$ leaves $P_2=(-1,0)$ 
below $\{F=0\}$ and above $\{G=0\}$, and the latter two level sets
are graphs of increasing functions of $V$. Furthermore, these functions
have infinite $V$-derivatives at $V=-1$, so that $\{F=0\}$, $\{G=0\}$, and
$\Gamma$ are all located above $\{C=1+V\}$ for $V\gtrsim-1$.
In particular $D(V,C)<0$ along $\{F=0\}$ and $\{G=0\}$ near $P_2$.

It follows from these properties, together with the fact that $F$ and $G$ are
proportional along the critical line $\{C=1+V\}$ (cf.\ \eq{non_obvious_reln}),
that the ``next'' intersection point $\bar P=(\bar V, \bar C)$ of $\{F=0\}$ and $\{G=0\}$ (strictly 
to the right of $P_2$), must be located on or above $\{C=1+V\}$. In particular,
for $-1<V<\bar V$, $\{F=0\}$ is located above $\{G=0\}$, which in turn is located above
$\{C=1+V\}$.

Next, it is straightforward to verify from 
\eq{G} that $G$ takes negative values at points within 
$\{(V,C)\,|\, -1\leq V<V_*, C>0\}$ which are located above $\{G=0\}$.
With $x<0$ we therefore have that any trajectory of \eq{V_sim2}-\eq{C_sim2}
located above $\{G=0\}$, is traversed from left to right in the $(V,C)$-plane
as $x$ increases. In particular, $V'(x)>0$ along $\{F=0\}$ for $-1<V<\bar V$.

Similarly, recalling that $\alpha>0$ by condition (C), an inspection of 
\eq{F} shows that $F(V,C)$ takes negative values at points within 
$\{(V,C)\,|\, V> -1, C>0\}$ which are located below $\{F=0\}$.
With $x<0$ it follows that any trajectory of \eq{V_sim2}-\eq{C_sim2}
located below $\{F=0\}$ and above $\{C=1+V\}$, is traversed upward in the $(V,C)$-plane
as $x$ increases. In particular, $C'(x)>0$ along $\{G=0\}$ for $-1<V<\bar V$.

With $C=C_F(V)$ and $C=C_G(V)$ parametrizing $\{F=0\}$ and $\{G=0\}$, respectively,
it  follows from these observations that the set 
\beq\label{trap_1}
	\mathcal T_1:=\{(V,C)\,|\, -1<V<\bar V, C_G(V)<C<C_F(V)\}
\eeq
is a trapping region for \eq{V_sim2}-\eq{C_sim2} (with $x<0$). Consequently,
the trajectory $\Gamma$, which starts out from $P_2$ between  $\{F=0\}$ and $\{G=0\}$,
must necessarily connect $P_2$ to $\bar P$ (cf.\ Figure \ref{V_C_plane}).

The arguments above show that trajectories of \eq{V_sim2}-\eq{C_sim2}
(with $x<0$ and increasing) enter $\mathcal T_1$ vertically along its lower 
boundary $\{G=0\}$ and horizontally along its upper boundary  $\{F=0\}$ (see 
Figure \ref{V_C_plane}). In particular, there are infinitely many trajectories reaching 
$\bar P$ from within $\mathcal T_1$.

In the next section we analyze the remaining critical points of \eq{CV_ode} in the 
upper half-plane (Section \ref{P4_P6_P8}). We then impose conditions guaranteeing that
$\bar P$ is located on $\{C=1+V\}$, and furthermore that this point 
is a nodal point for \eq{CV_ode} (Section \ref{P_6_node}). 
Nodality implies that there are infinitely many trajectories approaching $\bar P$ with 
the same {\em primary} slope, and we shall want $\Gamma$ to be one of these trajectories
(Section \ref{prim_slope}).

\section{Critical points in the upper half-plane}\label{upper_crit_points}
\subsection{The critical points $P_4, P_6, P_8$}\label{P4_P6_P8}
The critical points of \eq{CV_ode} in the upper half-plane are denoted 
$P_4=(V_4,C_4)$, $P_6=(V_6,C_6)$, and $P_8=(V_8,C_8)$ (following \cite{laz}). Of these, $P_6$ and 
$P_8$, when present, are necessarily located along the critical line $L_+=\{C=1+V\}$,
with $P_6$ located to the left of $P_8$ by definition.
These two points thus provide triple points where the trajectory $\Gamma$ can, potentially, cross $L_+$. 

In the previous section we argued that, with conditions (A)-(F) met, the trajectory 
$\Gamma$ necessarily connects $P_2=(-1,0)$ to the first point of intersection $\bar P$ 
between $\{F=0\}$ and $\{G=0\}$, as we move rightward from $P_2$. 
We shall want that $\bar P$ is the point $P_6\in L_+$. Obviously, a necessary condition 
for this is that $P_6$ is present, and this is  condition (G) below. Next, as $V_6<V_8$ 
by definition, the only remaining obstruction to the situation we seek is that the point 
$P_4$ is located to the left of $P_6$. To rule this out we  impose that $V_6<V_4$ 
(this is part condition (H) below, where for convenience we also require $V_4<V_8$;
the situation is then precisely as depicted in  in Figure \ref{V_C_plane}).

With these goals in mind we proceed to identify the 
critical points of \eq{CV_ode} in the upper half-plane. 
Some of the expressions in the following analysis are given in \cites{laz,jj}, 
and we will  simply quote the relevant formulae  from these works. 

These critical points are obtained  by first solving for $C^2$ in terms of $V$ 
from $G(V,C)=0$,  i.e.,
\beq\label{gV}
	C^2=C_G(V)^2:=\textstyle\frac{V(1+V)(\lambda+V)}{n(V-V_*)},
\eeq
and then using the result in $F(V,C)=0$, keeping in mind that we restrict attention to $C>0$.
This gives a cubic equation for $V$ which always has one real root
$V_4$ for any choice of $n,\gamma,\lambda,\kappa$, viz.
\[V_4=-\textstyle\frac{2\lambda}{2+n(\gamma-1)}.\] 
As a consequence, $P_4$ is present if and only if $C_G(V_4)>0$.

The remaining $V$-quadratic has real roots $V_6$ and $V_8$ only in certain 
cases. Provided that  $V_6$ and $V_8$ are real and, in addition, $C_G(V_6), C_G(V_8)>0$, 
we obtain the critical points $P_6$ and $P_8$.
It is convenient to set 
\[V_-:=V_6\qquad\text{and} \qquad  V_+:=V_8.\] 
A direct calculation (cf.\ \cite{jj}) shows that these are given by
\begin{align}
	&V_\pm=\textstyle\frac {1}{2 m \gamma} 
	\Big[(\gamma-2)\mu +\kappa-m\gamma\pm\sqrt{(\gamma-2)^2\mu^2
	-2[\gamma m(\gamma+2)-\kappa(\gamma-2)]\mu+(\gamma m+\kappa)^2}\Big],\label{V_pm}
\end{align}
where $m=n-1$, and $\mu=\lambda-1$. Note that $V_6=V_-<V_+=V_8$ by definition.
It follows from \eq{gV} and \eq{V_pm} that the critical points $P_6$ and $P_8$ are present if and 
only if, first, the radicand in \eq{V_pm} is positive, and second, $C_G(V_\pm)>0$. 
If so, the ODE \eq{CV_ode} has  the critical points 
\beq\label{P_6-P_9}
	P_6=(V_-,C_-)\qquad \text{and}\qquad P_8=(V_+,-C_+),
\eeq
where 
\[C_\pm=\sqrt{C_G(V_\pm)}.\] 
Note that $P_6$ and $P_8$ coalesce whenever $V_+=V_-$ (this does occur for certain values
of the parameters $n,\gamma,\lambda,\kappa$, cf.\ \cite{jj}).
A direct calculation reveals the non-obvious fact that $P_6,P_8\in L_+$, i.e.,
\beq\label{on_L_+-}
	C_\pm^2=(1+V_\pm)^2,
\eeq
so that all three functions $F(V,C)$, $G(V,C)$, $D(V,C)$ vanish at $P_6$ and $P_8$.

We now impose the condition that $P_6$ {\em is} present:
\begin{itemize}
	\item[(G)] $P_6$ is present, i.e., the radicand in \eq{V_pm} is positive and $C_G(V_-)>0$.
\end{itemize}
Note that \eq{g(W)} and \eq{gV} give $C_G(V)\equiv g(1+V)$.
With condition (F) in force, $g(W)$ is increasing so that $C_G(V)$ is an increasing function 
of $V$. It follows from this that condition (G) also implies that $P_8$ is present. 

It will be convenient for later arguments to have $P_4$ located between $P_6$ and $P_8$, 
and we therefore impose the further condition that 
\begin{itemize}
	\item[(H)] $V_-<V_4<V_+$.
\end{itemize}
According to the discussion at the beginning of the present section we thus have:
With conditions (A)-(H) met, the point denoted $\bar P$ above is the  critical point $P_6\in L_+$.
Furthermore, the trajectory $\Gamma$ is trapped between 
$\{F=0\}$ and $\{G=0\}$ for $-1<V<V_-$ and connects $P_2$ to $P_6$.

\subsection{Nodality of  $P_6$}\label{P_6_node}
As noted above we shall want $P_6$ to be a {\em node} for the ODE
\eq{CV_ode}, i.e., the linearization of \eq{CV_ode} at $P_6$  should have two real 
characteristic values of the same sign.

To formulate the conditions for this to be the case we recall some standard 
ODE theory (following the notation in \cite{laz}). 
The linearization of \eq{CV_ode} at any critical point $P_c=(V_c,C_c)$ is 
\beq\label{linzd_ode}
	\textstyle\frac{d\tilde C}{d\tilde V}=\frac{F_V\tilde V+F_C\tilde C}{G_V\tilde V+G_C\tilde C}
\eeq
where $\tilde V=V-V_c$, $\tilde C=C-C_c$, and the partials $F_V, F_C,G_V,G_C$ 
are evaluated at $P_c$. (It is assumed that these partials do not all vanish.)
The {\em Wronskian} $W$ and {\em discriminant} $R^2$ for the ODE \eq{CV_ode} at $P_c$ are defined by
\beq\label{wronsk}
	W:=F_CG_V-F_VG_C
\eeq
and
\beq\label{R_sqrd}
	R^2:=(F_C-G_V)^2+4F_VG_C\equiv (F_C+G_V)^2-4W.
\eeq
In what follows, $R$ denotes the positive square root of $R^2$ whenever the latter is positive.
Assuming for now that $R^2>0$, we define the characteristic directions (or, slopes)
\beq\label{L_12}
	L_{1,2}=\textstyle\frac{1}{2G_C}(F_C-G_V\pm R)
\eeq
and the characteristic values 
\beq\label{E_12}
	E_{1,2}=\textstyle\frac{1}{2G_C}(F_C+G_V\pm R),
\eeq
where the indices and signs are chosen so that 
\beq\label{Es}
	|E_1|<|E_2|,
\eeq
and with the understanding that the signs $\pm$ in \eq{L_12} and in \eq{E_12} agree.
We then have that integrals of \eq{linzd_ode} near $P_c$ 
approach one of the curves
\beq\label{gen_crit_point}
	(\tilde C-L_1\tilde V)^{E_1}=\text{constant}\times (\tilde C-L_2\tilde V)^{E_2}.
\eeq
$L_1$ and $L_2$ are  referred to as the {\em primary} and {\em secondary slopes} 
(or directions), respectively, at $P_c$. We observe that
\beq\label{W_prod}
	W\equiv E_1E_2G_C^2.
\eeq 
Thus, under the assumption that $R^2>0$, the critical point 
$P_c$ is a node (i.e., $\sgn E_1=\sgn E_2$) if and only if $W>0$. 
Finally, when $P_c$ is a node, \eq{gen_crit_point} implies that all but two of the integrals 
of \eq{CV_ode} near $P_c$ approach $P_c$ with the primary slope $L_1$. 
The exceptions are two integrals approaching $P_c$ with the secondary 
slope $L_2$.

Returning to $P_6$ (assumed to be present according to condition (F)), we proceed 
to analyze its nodality. Consider first the Wronskian $W_6$ at $P_6$. According to an 
elegant argument given in  \cite{laz} (p.\ 323), $W_6$ is given by
\beq\label{W_6_eqn}
	W_6=KC_6^2(V_6-V_4)(V_6-V_8)\equiv KC_6^2(V_--V_4)(V_--V_+),
\eeq
where the positive constant $K$ is given by
\beq\label{K}
	K=m(n(\gamma-1)+2).
\eeq
Since $V_-<V_+$ by definition (cf.\ \eq{V_pm}), condition (G) implies that $W_6>0$.
Therefore, under the conditions we have imposed above, $P_6$ is a node provided 
the discriminant $R_6^2$ is positive. We proceed to identify conditions guaranteeing that
this is the case.

As is evident from the formulae for $F$, $G$, theirs partials, and $V_6=V_-$,
the expression for $R_6^2$ is somewhat involved. Instead of writing out the general expression 
for $R^2_6$ in terms of the parameters $n,\gamma,\lambda,\kappa$, we provide 
formulae for the partials of $F$ and $G$ in terms of $V_-$, $C_-$, and then impose 
positivity of $R^2_6$ as a condition to be verified.

While the partials of $F$ and $G$ are readily calculated from \eq{F} and \eq{G}, it is convenient to 
simplify these by exploiting the fact that these are to be evaluated at the triple point $P_6$. In particular,
$D$ and $G$ both vanish at $P_6$ so that we have 
(dropping subscripts, evaluation at $P_6=(V_-,C_-)=(V_-,1+V_-)$ being understood)
\begin{align}
	C&=1+V\label{at_P_6_1}\\
	C^2&=\textstyle\frac{V(1+V)(\lambda+V)}{n(V-V_*)}.\label{at_P_6_2}
\end{align}
Note that these give
\beq\label{at_P_6_3}
	nC(V-V_*)=V(\lambda+V).
\eeq
Using the relations \eq{at_P_6_1}-\eq{at_P_6_3} in the formulae for the partial of $F$ 
and $G$, together with the fact that $F$ vanishes at $P_6$, then yield the expressions
\begin{align}
	F_C&=2C(1+\alpha+V)\label{F_C}\\
	F_V&=C[k_2-\alpha-2k_1(1+V)]\label{F_V}\\
	G_C&=2nC(V-V_*)\label{G_C}\\
	G_V&=C[n(1+V_*)-2V-\lambda].\label{G_V}
\end{align}
We finally state the positivity of the discriminant at $P_6$ as a condition to be verified:
\begin{itemize}
	\item[(I)] The discriminant $R^2$ given in \eq{R_sqrd}, when evaluated at $P_6=(V_-,C_-)$ 
	using \eq{F_C}-\eq{G_V}, is strictly positive.
\end{itemize}
Assuming conditions (A)-(I), we proceed to determine the signs to be used in \eq{E_12}, and thus 
the primary and secondary slopes $L_{1,2}$ at $P_6$. The requirement in \eq{Es} amounts to
\beq\label{signs1}
	|2G_C||E_1|=|F_C+G_V\mp R|<|F_C+G_V\pm R|=|2G_C||E_2|.
\eeq
Nodality of $P_6$ means that $W>0$, so that \eq{R_sqrd} gives 
\beq\label{signs2}
	R<|F_C+G_V|.
\eeq
Also, \eq{F_C} and \eq{G_V}, together with the expressions for $\alpha$ and $V_*$, 
give
\[F_C+G_V=(\kappa+n+1-\mu)C.\]
Since $C_6>0$ it follows from condition (A) that $F_C+G_V>0$ at $P_6$. 
According to \eq{signs2} we thus have $0<R<F_C+G_V$, and it follows from \eq{signs1}
that the minus sign in \eq{E_12} corresponds to $E_1$, and the plus sign to $E_2$.
As a consequence, the primary and secondary slopes $L_{1,2}$ at $P_6$ are given as
\beq\label{primry_secdry_slopes}
	L_1=\textstyle\frac{1}{2G_C}(F_C-G_V- R),\qquad L_2=\textstyle\frac{1}{2G_C}(F_C-G_V+R),
\eeq
respectively.

\subsection{The approach of $\Gamma$ to $P_6$}\label{prim_slope}
As noted above, we shall want $\Gamma$ to approach $P_6$ with the {\em primary}
slope $L_1$ of \eq{V_sim2}-\eq{C_sim2} at $P_6$, so that there are infinitely many other trajectories 
approaching with the same slope. This will be exploited below in arguing that $\Gamma$ can be 
continued to connect $P_6$ to the origin $P_1$ in the $(V,C)$-plane (Section \ref{conclude}). 

Recall that, with conditions (A)-(I) in force, we know that $\Gamma$ 
approaches $P_6$ from within the  trapping region $\mathcal T_1$ in \eq{trap_1} (see
discussion following Eqn.\ \eq{trap_1}). Also, $\mathcal T_1$ is bounded above and below by the 
levels sets $\{F=0\}$ and $\{G=0\}$, respectively. Therefore,
a simple condition guaranteeing that $\Gamma$ approaches $P_6$ 
with the primary slope $L_1$, is that $L_1$ is intermediate 
to the slopes of $\{F=0\}$ and $\{G=0\}$ at $P_6$, while $L_2$ is not.
The slopes of the latter curves are given by $-\frac{F_V}{F_C}$ and $-\frac{G_V}{G_C}$, 
respectively, and we therefore impose the (final) condition:
\begin{itemize}
\item[(J)] At $P_6$ there holds
	\beq\label{rel_locns}
		L_2<-\textstyle\frac{F_V}{F_C}<L_1<-\textstyle\frac{G_V}{G_C},
	\eeq
	where $L_1,L_2$ are given in \eq{primry_secdry_slopes}.
\end{itemize}

\section{Concluding argument}\label{conclude}
With all conditions (A)-(J) in force, we can complete the argument for
how $\Gamma$ can be continued beyond $P_6$ to reach the origin. 
First, $\Gamma$ arrives at $P_6$ with the primary slope $L_1$. Since $P_6$ is a node 
there are infinitely many trajectories leaving $P_6$ with slope $L_1$ and moving rightward
as $x$ increases. 
It follows from \eq{rel_locns} that, near $P_2$, these trajectories are located below $\{G=0\}$ 
and above $\{F=0\}$. Infinitely many of these trajectories will meet $\{F=0\}$, and they 
do so with vanishing slope. (Infinitely many others will meet $\{G=0\}$
vertically; in the concrete cases we consider below, there is a separatrix connecting $P_6$ to 
a saddle at $P_4$.)

We now choose any one of the trajectories that cross $\{F=0\}$ near (and to the right of) $P_6$
as our continuation of $\Gamma$ beyond $P_6$. We claim that the resulting trajectory, i.e., 
$\Gamma=\{(V(x),C(x))\,|\, x_0<x<0\}$, connects $P_6$, and hence $P_2$, to the origin 
$P_1$ as $x\uparrow0$.

To verify the claim, let $P_0=(V_0,C_0)$ denote the point where $\Gamma$ crosses $\{F=0\}$, and let 
$\hat V$ be the unique positive $V$-value with $C_G(\hat V)=C_0$. 
(Note that $\hat V>0$ is uniquely determined since the function $C_G(V)$ is increasing on $[0,\infty)$).
By tracking the signs of $F(V,C), G(V,C)$ it follows that the set
\beq\label{trap_2}
	\mathcal T_2:=\{(V,C)\,|\, V_0<V<\hat V, 0<C<C_0\}.
\eeq
is a trapping region for \eq{V_sim2}-\eq{C_sim2} (with $x<0$); see Figure \ref{V_C_plane}. 
The argument for this is straightforward thanks to the fact that both $\{F=0\}$ and $\{G=0\}$ 
are graphs of {\em increasing} functions of $V$, which is ensured by conditions (D)-(F). 
In particular, $\{F=0\}$ and $\{G=0\}$ meet $\partial \mathcal T_2$ only at $P_0$ and at 
$\hat P:=(\hat V, C_0)$, respectively. It follows that $\mathcal T_2$ contains no critical point 
for \eq{CV_ode}, and any trajectory  within $\mathcal T_2$ must therefore tend to the 
unique critical point $P_1=(0,0)\in\partial \mathcal T_2$. 

A further inspection of  \eq{V_sim2}-\eq{C_sim2} shows that 
\begin{itemize}
	\item{} $C'(x)<0<V'(x)$ whenever $(V(x),C(x))\in\mathcal T_2^{left}$, where
	\[\mathcal T_2^{left}:=\{(V,C)\,|\, V_0<V<0, 0<C<C_0\}\cup\{(V,C)\,|\, 0\leq V<V_0, C_G(V)<C<C_0\};\]
	\item{} $C'(x),V'(x)<0$ whenever $(V(x),C(x))\in\mathcal T_2^{right}$, where
	\[\mathcal T_2^{right}:=\{(V,C)\,|\, 0<V<0, 0<C<C_G(V)\}.\]
\end{itemize}
Therefore, after $\Gamma$ has entered $\mathcal T_2$ at its top-left corner, it will move in a 
South-East direction within $\mathcal T_2^{left}$. One of two things can now occur: Either $\Gamma$
remains within $\mathcal T_2^{left}$ until it reaches $P_1$, or it first crosses vertically into 
$\mathcal T_2^{right}$ along $\{G=0\}$, and then continues on to $P_1$. (In our numerical tests we have not observed the latter behavior.)
Either way, $\Gamma$ connects $P_2$ to $P_1$, via $P_6$. 

Finally, it was argued in Section \ref{P1} that whenever a solution $(V(x),C(x))$ of 
\eq{V_sim2}-\eq{C_sim2} approaches $P_1$ from within $\{C>0\}$, then this must occur as $x\uparrow0$. 
It follows that $\Gamma$ is the trajectory of a solution to \eq{V_sim2}-\eq{C_sim2} 
which is defined for all $x\in(x_0,0)$. 

Defining the physical flow variables via \eq{sim_vars}, we therefore obtain an
Euler flow within the space-time (fluid) region $\{(t,x)\,|\, t<0, |x|>r_0(t)\}$, where $r_0(t)$ is defined in \eq{r_0}.
We note that we have in fact constructed infinitely many examples of such cavity flows.
More precisely, the part of $\Gamma$ connecting $P_2$ to $P_6$ is uniquely determined
by the requirement that the pressure should vanish along the vacuum interface; 
see concluding remark in Section \ref{P2}. However, with our setup above ($P_6$ being 
a node), we have an infinite number of possible trajectories connecting $P_6$ to $P_1$.

Modulo the conditions (A)-(J), this completes the construction of non-isentropic, radial, and 
self-similar cavity Euler flows \eq{sim_vars} as described by Theorem \ref{thm}.

\begin{figure}
	\centering
	\includegraphics[width=10cm,height=10cm]{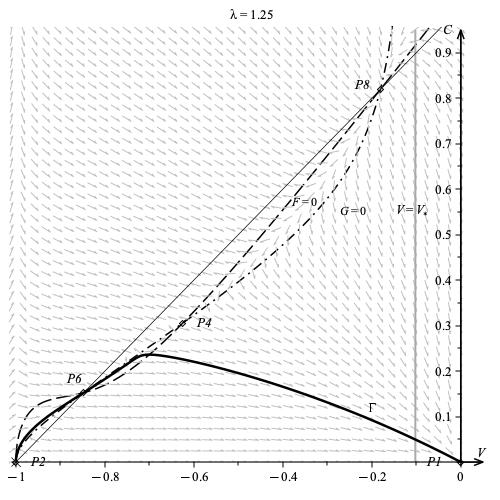}
	\caption{Phase portrait for the case $n=3$, $\gamma=\frac{5}{3}$, $\lambda=1.25$, $\kappa=-0.01$.}\label{case_1}.
\end{figure} 

\begin{figure}
	\centering
	\includegraphics[width=10cm,height=10cm]{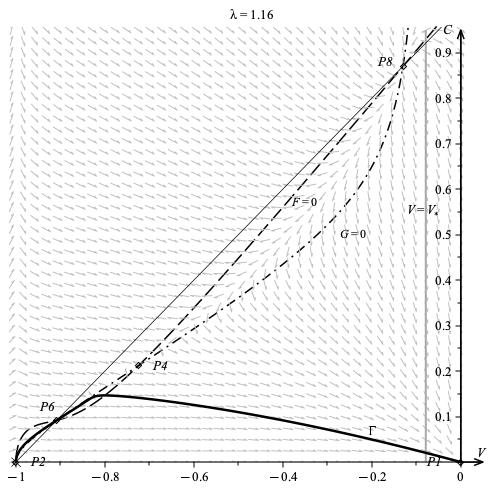}
	\caption{Phase portrait for the case $n=3$, $\gamma=\frac{7}{5}$, $\lambda=1.16$, $\kappa=-0.01$.}\label{case_2}.
\end{figure} 

\begin{figure}
	\centering
	\includegraphics[width=10cm,height=10cm]{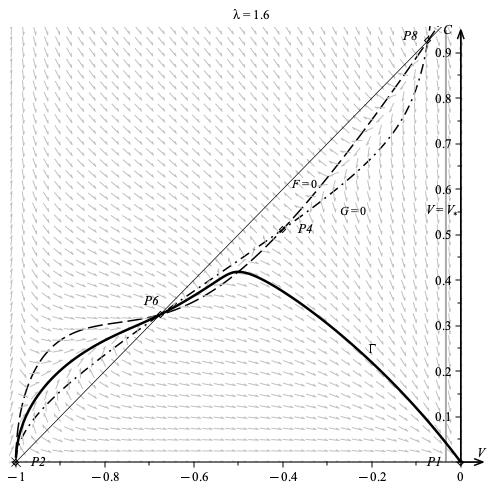}
	\caption{Phase portrait for the case $n=3$, $\gamma=3$, $\lambda=1.6$, $\kappa=0.9$.}\label{case_3}.
\end{figure} 

\section{Concrete choices of parameters verifying conditions (A)-(J)}\label{par_choices}
The verification of (A)-(J) for specific choices of the parameters 
$n,\gamma,\lambda,\kappa$
involves tedious but straightforward calculations. Below we list several physically 
relevant cases, in both two and three dimensions, where all 
ten conditions have been verified to hold. These are:
\begin{enumerate}
	\item $n=3$, $\gamma=\frac{5}{3}$, $\lambda=1.25$, $\kappa=-0.01$;
	\item $n=3$, $\gamma=\frac{7}{5}$, $\lambda=1.16$, $\kappa=-0.01$; 
	\item $n=3$, $\gamma=3$, $\lambda=1.6$, $\kappa=0.9$; 
	\item $n=2$, $\gamma=\frac{5}{3}$, $\lambda=1.09$, $\kappa=-0.01$; 
	\item $n=2$, $\gamma=\frac{7}{5}$, $\lambda=1.06$, $\kappa=-0.01$;
	\item $n=2$, $\gamma=3$, $\lambda=1.28$, $\kappa=-0.01$; 
\end{enumerate}
Maple plots of the phase portraits displaying the trajectory $\Gamma$, the level sets $\{F=0\}$, $\{G=0\}$,
as well as directional field for \eq{V_sim2}-\eq{C_sim2} (for $x<0$), are given 
for the 3-dimensional cases in Figures \ref{case_1}-\ref{case_3}. In each figure 
$\Gamma$ is the thick solid curve, $\{F=0\}$ and $\{G=0\}$ are dashed and dashed-dotted,
respectively; the vertical grey line is the asymptote $V=V_*$ of $\{G=0\}$. 
The critical line $\{C=1+V\}$ and the relevant critical points $P_1$, $P_2$, $P_4$, $P_6$,
and $P_8$ are also displayed.

We finally note that the conditions (A)-(J) all take the form of strict inequalities involving 
continuous functions of the parameters $\gamma,\lambda,\kappa$. Thus, for each of the 
cases (1)-(6) listed above, (A)-(J) will be satisfied for all values of $\gamma,\lambda,\kappa$ sufficiently
close to those listed.



\end{document}